\documentclass[12pt]{article}

\usepackage{lineno}
\usepackage[T1]{fontenc}
\usepackage{graphicx}
\usepackage{subfigure}
\usepackage{amsthm,mathrsfs}
\usepackage{amsmath}
\usepackage{amsfonts}
\usepackage{bbm}
\usepackage{float}
\usepackage{url}
\usepackage{color}
\usepackage{bm,bbm}
\usepackage{times}
\usepackage{algorithm,algpseudocode}
\usepackage{algcompatible}
\usepackage{caption,dsfont}
\usepackage{booktabs}
\usepackage{graphicx}
\usepackage{enumerate}
\usepackage{natbib}
\usepackage{hyperref}
\usepackage{psfrag,scalefnt}
\usepackage{setspace,hyperref}

\def\orcid#1{\kern .08em\href{https://orcid.org/#1}{\includegraphics[keepaspectratio,width=0.7em]{orcid.pdf}}}

\newtheorem{theorem}{Theorem}[section]

\newtheorem{corollary}[theorem]{Corollary}
\newtheorem{proposition}[theorem]{Proposition}
\newtheorem{definition}[theorem]{Definition}
\newtheorem{remark}[theorem]{Remark}
\theoremstyle{definition}

\begin{document}
	
%
\title{\bf On some properties of the bimodal normal distribution \\
	and its bivariate version}

\date{\today}
%
%

\author{
{ Roberto Vila$^1$ }, \
{ Helton Saulo$^1$ } \ and \
{Jamer Roldan$^2$ }
\\[0.15cm]
{\normalsize
$^1$Department of Statistics, Universidade de Brasília, 70910-900, Brasília, Brazil}
\\
{\normalsize
$^2$Department of Mathematics, Instituto Federal de Goiás, 72876-601, Goiás, Brazil}
} 
%

\maketitle 
%
\begin{abstract}
In this work, we derive some novel properties of the bimodal normal distribution. Some of its mathematical properties are examined. We provide a formal proof for the bimodality and assess identifiability. We then discuss the maximum likelihood estimates as well as the existence of these estimates, and also some asymptotic properties of the estimator of the parameter that controls the bimodality. A bivariate version of the BN distribution is derived and some characteristics such as covariance and correlation are analyzed. We study stationarity and ergodicity and a triangular array central limit theorem. Finally, a Monte Carlo study is carried out for evaluating the performance of the maximum likelihood estimates.

\end{abstract}

\paragraph{Keywords:}Bimodality; Identifiability; Bivariate distribution; Stationarity; Ergodicity; Central limit theorem.

\maketitle

\section{Introduction}
Bimodal distributions play an important role in the applied statistical literature; see, for example, \cite{eugeneetal:02} and \cite{hassanelbassiouni:16}. The use of mixture-free bimodal distributions is very important as often real-world data are better modeled by these models, and in general, mixtures of distributions may suffer from identifiability problems in the parameter estimation; see \cite{vilaetal:20}. Recently, \cite{gomezdenizetal:21} introduced a family of continuous distributions appropriate to describe the behavior of bimodal data. This family can accommodate any symmetric distribution and for the normal case, the random variable $X$ has the following probability density function (PDF) 
\begin{align}\label{density-family-BN}
 f_{\alpha,\zeta}(x)=\sqrt{2\pi}\,\text{sech}(\zeta\alpha)\phi(\alpha)\phi(x)\text{cosh}[\alpha(x-\zeta)], \quad x\in\mathbb{R},
\end{align}
where $\zeta\in\mathbb{R}$ and $\alpha\in\mathbb{R}$ are shape and location parameters, respectively, $\phi(x)$ is the standard normal PDF, and $\text{sech}(z)=1/\text{cosh}(z)$, with $\text{cosh}(z)=[\exp(z)+\exp(-z)]/2$. The parameter $\zeta$ in \eqref{density-family-BN} controls the skewness and the parameter $\alpha$ is related to the bimodality; see \cite{gomezdenizetal:21}.

In this work, we derive some novel properties of a special case of Equation \eqref{density-family-BN}, more specifically when $\zeta = 0$. Then, we say that a real-valued random variable $X$ has a uni- or bimodal normal (BN) distribution with parameter vector parameter $\boldsymbol{\theta}=(\mu,\sigma,\alpha)$, $\mu\in\mathbb{R}$, $\sigma>0$, $\alpha\in \mathbb{R}$, denoted by $X\sim \text{BN}(\boldsymbol{\theta})$, if its PDF is given by
\begin{align}\label{density-BN}
f(x;\boldsymbol{\theta})
=
\dfrac{1}{\sqrt{2\pi\sigma^2}}\,
\exp\biggl[-{1\over 2}\, {\biggl({x-\mu\over \sigma}\biggr)^2 - {\alpha^2\over 2}}\,\biggr] 
\cosh\biggl[\alpha\biggl({ x-\mu\over \sigma}\biggr)\biggr],
\quad x\in\mathbb{R},
\end{align}
where $\mu$ is a location parameter, $\sigma$ is a scale parameter, and $\alpha$ is a parameter that controls the uni- or bimodality of the distribution. When $\alpha$ approaches 0 (i.e. $\vert\alpha\vert\leq 1$) the distribution becomes unimodal and when $\alpha$ grows (i.e. $\vert\alpha\vert>1$) the bimodality becomes more accentuated. When $\alpha=0$ we have the known normal distribution. For more details, see Theorem \ref{Main Theorem}.

The rest of this paper proceeds as follows. In Section \ref{sec:pre:prop}, we briefly describe some preliminary properties, including the behaviour of the density and hazard functions, median,  moment generating function, mean, variance, among others. In Section \ref{sec:bimodality}, we obtain some results on the bimodality property of the BN distribution, and the stochastic representation and moments are derived in Section \ref{sec:stoch:mom}. In Section \ref{Identifiability}, we study some aspects of identifiability. In Section \ref{sec:asymptot}, we discuss maximum likelihood (ML) estimation, existence of the ML estimates, and some asymptotic properties of the ML estimator (MLE) of $\alpha$. A bivariate version of the BN distribution is derived and some characteristics such as covariance and correlation are analyzed in Section \ref{sec:bivariate}. In Section \ref{sec:stat:ergo}, the concepts of stationarity and ergodicity of a BN random process are studied. Ergodicity is an important ingredient to study functions of the distributional characteristics of the process when we have only one realization. We find out that the BN random process is not stationary. This result allows us to study, in Section \ref{sec:triangular}, the triangular array central limit theorem, which is of vital importance in statistics. In Section \ref{sec:simulation}, we carry out Monte Carlo simulations. Finally, in Section \ref{sec:concluremarks}, we discuss conclusions.

\section{Preliminary properties}\label{sec:pre:prop}
Let $X\sim \text{BN}(\boldsymbol{\theta})$ with PDF $f(x;\boldsymbol{\theta})$ given in \eqref{density-BN}. Then, the behavior of $f(x;\boldsymbol{\theta})$ with $x\to 0$ or $x \to \pm\infty$ is as follows:
\begin{align}\label{limits-density}
\lim_{x\to 0} f(x;\boldsymbol{\theta})
=
\sqrt{2\pi}\, 
\phi_{\mu,\sigma^2}(0) 
\phi(\alpha) 
\cosh\Big({\alpha\mu\over\sigma}\Big)
\quad \text{and} \quad
\lim_{x\to \pm\infty} f(x;\boldsymbol{\theta})=0,
\end{align}
where $\phi_{\mu,\sigma^2}(x)$ is the PDF of the normal distribution with mean $\mu$ and variance $\sigma^2$, and we denote $\phi(x)$ instead $\phi_{0,1}(x)$.

It is verified that the cumulative distribution function (CDF) of $X\sim \text{BN}(\boldsymbol{\theta})$ is given by
\begin{align}\label{CDF-Bigaussian}
F(x;\boldsymbol{\theta})
=
{1\over 4}\, 
\left[2+{\rm erf}\biggl({x-\mu-\alpha\sigma\over \sigma\sqrt{2}}\biggr)
+
{\rm erf}\biggl({x-\mu+\alpha\sigma\over \sigma\sqrt{2} }\biggr)\right],
\end{align}
where ${\rm erf}(x)=2\int_0^{x} \exp(-t^2)\, {\rm d}t/\sqrt{\pi}$ is the error function. Note that $\lim_{\alpha\to 0}F(x;0,1,\alpha)=(1/2)[1+{\rm erf}({x/\sqrt{2}})]=\Phi(x)$, where $\Phi(x)$ is the CDF of the normal distribution.

The hazard function $h(x;\boldsymbol{\theta})=f(x;\boldsymbol{\theta})/[1-F(x;\boldsymbol{\theta})]$ has the following behavior with $x\to 0$ or $x \to \pm\infty$:
\begin{align*}
\lim_{x\to 0} h(x;\boldsymbol{\theta})
=
{4\sqrt{2\pi}\, 
\phi_{\mu,\sigma^2}(0) 
\phi(\alpha) 
\cosh({\alpha\mu/\sigma})
\over 
2- {\rm erf}\big({-\mu-\alpha\sigma\over\sigma\sqrt{2}}\big)
-
{\rm erf}\big({-\mu+\alpha\sigma\over\sigma\sqrt{2} }\big)
},
\quad 
\lim_{x\to -\infty} h(x;\boldsymbol{\theta})=0
\quad \text{and} \quad
\lim_{x\to +\infty} h(x;\boldsymbol{\theta})=+\infty.
\end{align*}
From the above limits it can be concluded that the hazard function is not a decreasing function.

\smallskip 
A routine calculation shows that, if $X\sim \text{BN}(\boldsymbol{\theta})$,
\begin{enumerate}
\item[(P.1)] (Density) The random variable $Z=(X-\mu)/\sigma$, where $\mu\in\mathbb{R}$ and $\sigma>0$, has PDF given by

$f(z;0,1,\alpha)
=
\dfrac{1}{\sqrt{2\pi}}
\exp\bigl(-{z^2 + \alpha^2\over 2}\bigr)
\cosh(\alpha z)$,
\ $z\in\mathbb{R}$.

That is, $Z\sim \text{BN}(0,1,\alpha)$;
\item[(P.2)] If $f$ is a Borel measurable function then

$\mathbb{E}\big[f({X-\mu\over\sigma})\big]=\exp\bigl(-{\alpha^2\over 2}\bigr)\mathbb{E}_{\Phi}[f(Z)\cosh(\alpha Z)]$, $Z\sim N(0,1)$,

where $\mathbb{E}_{\Phi}(\cdot)$ denotes the expectation with respect to distribution function $\Phi$;
\item[(P.3)] (Symmetry) $f(\mu-x;\boldsymbol{\theta})=f(\mu+x;\boldsymbol{\theta})$ for all real numbers $x$;
\item[(P.4)] (Median) The median $m$ satisfies:
${\rm erf}\big({m-\mu-\alpha\sigma\over\sigma\sqrt{2}}\big)={\rm erf}\big({-m+\mu-\alpha\sigma\over \sigma\sqrt{2}}\big)$. Then $m=\mu$;
\item[(P.5)] (Moment generating function) 
$M_X(t)=  
\exp\big(\mu t+{1\over 2}\, \sigma^2 t^2\big) \cosh(\alpha\sigma t)$, $t\in\mathbb{R}$;
\item[(P.6)] (CF) $\phi_X(t)= \exp\big(i\mu t-{1\over 2}\, \sigma^2 t^2\big) \cosh(i\alpha\sigma t)$, $t\in\mathbb{R}$;
\item[(P.7)] (Mean) $\mathbb{E}(X)=\mu$; 
\item[(P.8)] (Variance) ${\rm Var}(X)=\sigma^2(1+\alpha^2)$;
\item[(P.9)] (Skewness) $v=0$. 

That is, the distribution is approximately symmetrical;
\item[(P.10)] (Kurtosis) $\kappa=\alpha^2(\alpha^2+6)+3$;
\item[(P.11)] (Mean absolute deviation) 
${\rm MAD}=
\big[2\phi(\alpha)+\alpha{\rm erf}\big({\alpha\over \sqrt{2}}\big)\big]\sigma$;
\item[(P.12)] (Shannon entropy) $H(X)=\log(\sqrt{2\pi\sigma^2})+{2\alpha^2+1\over 2}-{\exp(-\alpha^2/2)\over 2}\,[\exp(2\alpha^2)+1]$.
\end{enumerate}

\section{Uni- or bimodality of the BN distribution}\label{sec:bimodality}

\begin{theorem}[Uni- or bimodality]\label{Main Theorem}
The PDF of the BN distribution \eqref{density-BN} is unimodal when $\vert\alpha\vert\leq 1$ and is bimodal when $\vert\alpha\vert> 1$.
\end{theorem}
\begin{proof}
	Let us suppose that $\alpha\neq 0$ because for the case $\alpha= 0$ the unimodality is well known.
	
The derivative of $f(x;\boldsymbol{\theta})$ with respect to $x$ is:
\begin{align*}
f'(x;\boldsymbol{\theta})
&=
{f(x;\boldsymbol{\theta})\over\sigma}\,
\biggl\{\alpha\tanh\biggl[\alpha\biggl({x-\mu\over \sigma}\biggr)\biggr]- \biggl({x-\mu\over \sigma}\biggr)\biggr\}.
\end{align*}
Then, $f'(x;\boldsymbol{\theta})=0$ if and only if
\begin{align}\label{derivative}
\tanh\biggl[\alpha\biggl({x-\mu\over \sigma}\biggr)\biggr]
=
{x-\mu\over \alpha\sigma}.
\end{align}

Let $g(x;\boldsymbol{\theta})=\tanh[\alpha (x-\mu)/\sigma]-{(x-\mu)/(\alpha\sigma)}$.
Note that, for all $\alpha\neq 0$, $x=\mu$ is a root of $g(x;\boldsymbol{\theta})$. In what follows we divided the proof in two steps.

\smallskip
{\sc First step: proving unimodality.}
Note that,  $g'(x;\boldsymbol{\theta})=(1/\sigma)\{\alpha{\rm sech}^2[\alpha(x-\mu)/\sigma]-1/\alpha\}<0$ on  $(-\infty,+\infty)$ when $0<\alpha\leq 1$, and $g'(x;\boldsymbol{\theta})>0$ on $(-\infty,+\infty)$ when $-1\leq \alpha<0$, because ${\rm sech}^2(x)\leq 1$.

Since the function $g(x;\boldsymbol{\theta})$ has opposite signs at the extremes of the interval (i.e., $\lim_{x\to -\infty}g(x;\boldsymbol{\theta})=+\infty$, $\lim_{x\to +\infty}g(x;\boldsymbol{\theta})=-\infty$ when $0<\alpha\leq 1$, and $\lim_{x\to -\infty}g(x;\boldsymbol{\theta})=-\infty$, $\lim_{x\to +\infty}g(x;\boldsymbol{\theta})=+\infty$ when $-1\leq \alpha<0$) and is monotonic, it will have a single zero at $x=\mu$. Then, since $\lim_{x\to \pm\infty} f(x;\boldsymbol{\theta})\stackrel{\eqref{limits-density}}{=}0$, the unimodality of the BN distribution \eqref{density-BN} is guaranteed.


\smallskip
{\sc Second step: proving bimodality.}
Without loss of generality, now we assume that $\alpha>1$ because the other case  $\alpha<-1$ is verified using similar arguments. For this case, note that $g(x;\boldsymbol{\theta})>0$ when $x\leq \mu-\sigma\alpha$ and $g(x;\boldsymbol{\theta})<0$ when $x\geq \mu+\sigma\alpha$. Then, there is no root of $g(x;\boldsymbol{\theta})$ outside of the interval $(\mu-\sigma\alpha,\mu+\sigma\alpha)$.

Using Intermediate value theorem, $g(\mu-\sigma\alpha;\boldsymbol{\theta})=1-\tanh(\alpha^2)>0$, $\varepsilon^-=\lim_{x\to \mu^-}g(x;\boldsymbol{\theta})<0$, 
and
$\varepsilon^+=\lim_{x\to \mu^+}g(x;\boldsymbol{\theta})>0$, $g(\alpha;\boldsymbol{\theta})=\tanh(\alpha^2)-1<0$, thus, there are $c_1\in (\mu-\sigma\alpha,\varepsilon^-)$ and $c_3\in (\varepsilon^+,\mu+\sigma\alpha)$: $g(c_i;\boldsymbol{\theta})=0$ for $i=1,3$.

Now we prove uniqueness of root on  $(\mu-\sigma\alpha,\varepsilon^-)$. Indeed,
assume  that $g(x;\boldsymbol{\theta})$ has two solutions $g(a;\boldsymbol{\theta})=g(b;\boldsymbol{\theta})=0$, $\mu-\sigma\alpha<a<b<\varepsilon^-$, then according to Rolle's theorem there is $c^*\in(a,b)$: $g'(c^*;\boldsymbol{\theta})=0$. But $g'(x;\boldsymbol{\theta})=(1/\sigma)[\alpha{\rm sech}^2(\alpha x)-{1/\alpha}]<0$ on $(\mu-\sigma\alpha,\varepsilon^-)$ with $\alpha>1$, and has no solutions, contradiction. Therefore, $g(x;\boldsymbol{\theta})$ has exactly one real solution on $(\mu-\sigma\alpha,\varepsilon^-)$. Similarly, it is verified that on $(\varepsilon^+,\mu+\sigma\alpha)$,  $g(x;\boldsymbol{\theta})$ has exactly one real solution.

In other words, for $\alpha>1$, $g(x;\boldsymbol{\theta})$ has exactly three real roots, denoted by $x_1,x_2, x_3$,  such that $x_1<x_2=\mu<x_3$. Finally, since $\lim_{x\to \pm\infty} f(x;\boldsymbol{\theta})\stackrel{\eqref{limits-density}}{=}0$, the  bimodality of the BN distribution \eqref{density-BN} follows.
\end{proof}

\begin{remark}
The modes of the BN distribution belong to the interval $(\mu-\sigma\alpha,\mu+\sigma\alpha)$.

By symmetry, there is $\delta=\delta(\sigma,\alpha)\in(0,\sigma\alpha)$ so that $x_1=\mu-\delta$ and $x_3=\mu+\delta$.
	
Moreover, when $\vert\alpha\vert>1$ and $\vert x\vert$ is sufficiently large, the modes of the BN distribution  are given by $x_1\approx \mu-\sigma\alpha$ and $x_3\approx \mu+\sigma\alpha$, because $\lim_{x\to\pm\infty}\tanh[\alpha (x-\mu)/\sigma]=\pm 1$.
\end{remark}

\begin{corollary}
The modal point $x_0=x_0(\boldsymbol{\theta})$ is a non-decreasing function of $\mu$ whenever $\vert\alpha\vert \leq 1$.
\end{corollary}
\begin{proof}
By \eqref{derivative}, a modal point $x_0$ of BN distribution satisfies
\begin{align}\label{modal-point}
x_0=\alpha\sigma \tanh\biggl[\alpha\biggl({x_0-\mu\over\sigma}\biggr)\biggr]+\mu.
\end{align}
Differentiating $x_0$ with respect to $\mu$ gives
\begin{align*}
	{\partial x_0\over\partial\mu}
	=
	1-\alpha^2{\rm sech}^2\biggl[\alpha\biggl({x_0-\mu\over\sigma}\biggr)\biggr]\geq 0,
\end{align*}
whenever $\vert\alpha\vert \leq 1$.

Hence $x_0$ is a non-decreasing function of $\mu$.
\end{proof}

\begin{corollary}
The modal point $x_0=x_0(\boldsymbol{\theta})$ is a non-decreasing function of $\sigma$ (resp. of $\alpha$) whenever $x_0 \geq \mu$ and a non-increasing function  of $\sigma$ (resp. of $\alpha$) whenever $x_0< \mu$.
\end{corollary}
\begin{proof}
Differentiating $x_0$ in \eqref{modal-point} with respect to $\sigma$ and $\alpha$ gives
\begin{align*}
{\partial x_0\over\partial\sigma}
=
\alpha
{\rm tanh}\biggl[\alpha\biggl({x_0-\mu\over\sigma}\biggr)\biggr]
-
\alpha^2\biggl({x_0-\mu\over\sigma}\biggr)
{\rm sech}^2\biggl[\alpha\biggl({x_0-\mu\over\sigma}\biggr)\biggr]
\end{align*}
and
\begin{align*}
{\partial x_0\over\partial\sigma}
=
\sigma\biggl\{
{\rm tanh}\biggl[\alpha\biggl({x_0-\mu\over\sigma}\biggr)\biggr]
+
\alpha\biggl({x_0-\mu\over\sigma}\biggr)
{\rm sech}^2\biggl[\alpha\biggl({x_0-\mu\over\sigma}\biggr)\biggr]
\biggr\}.
\end{align*}
From the above equations it follows that ${\partial x_0/\partial\sigma}\geq 0$ (resp. ${\partial x_0/\partial\alpha}\geq 0$) whenever $x_0 \geq \mu$ and ${\partial x_0/\partial\sigma}< 0$ (resp. ${\partial x_0/\partial\alpha}< 0$) whenever $x_0< \mu$.
\end{proof}

\section{Stochastic representation and moments}\label{sec:stoch:mom}

\begin{proposition}[Stochastic representation]\label{Stochastic representation}
Suppose $Z_{\mu,\sigma^2}$ has a normal distribution with expected value $\mu$ and variance $\sigma^2$. Let $W$ have the Bernoulli distribution, so that $W=\alpha\sigma$ or $W=-\alpha\sigma$, each with probability $1/2$, and assume $W$ is independent of $Z_{\mu,\sigma^2}$.
If $X=Z_{\mu,\sigma^2}+W$
then $X\sim \text{BN}(\boldsymbol{\theta})$.

Conversely, if $X\sim \text{BN}(\boldsymbol{\theta})$ then $X=Z_{\mu,\sigma^2}+W$.
\end{proposition}
\begin{proof} 
By Law of total probability and by independence, we get 
\begin{align*}
\mathbb{P}(X\leq x)
&=
\mathbb{P}(Z_{\mu,\sigma^2}+\alpha\sigma\leq x)\, \mathbb{P}(W=\alpha\sigma)+\mathbb{P}(Z_{\mu,\sigma^2}-\alpha\sigma\leq x)\, \mathbb{P}(W=-\alpha\sigma)
\\[0,2cm]
&=
\mathbb{P}(Z_{\mu,\sigma^2}+\alpha\sigma\leq x)\, {1\over 2}+\mathbb{P}(Z_{\mu,\sigma^2}-\alpha\sigma\leq x)\, {1\over 2}
\\[0,2cm]
&=
\Phi\biggl({x-\mu-\alpha\sigma\over\sigma}\biggr)\, {1\over 2}
+
\Phi\biggl({x-\mu+\alpha\sigma\over\sigma}\biggr)\, {1\over 2}.
\end{align*}
By using the identity $\Phi(x)=(1/2)[1+{\rm erf}({x/\sqrt{2}})]$, the above expression is equal to
\begin{align*}
{1\over 4}\, 
\left[2+{\rm erf}\biggl({x-\mu-\alpha\sigma\over \sigma\sqrt{2}}\biggr)
+
{\rm erf}\biggl({x-\mu+\alpha\sigma\over \sigma\sqrt{2} }\biggr)\right]
\stackrel{\eqref{CDF-Bigaussian}}{=}
F(x;\boldsymbol{\theta}),
\quad x\in\mathbb{R}.
\end{align*}

%
%
Then we have complete the proof.
\end{proof}


\begin{proposition}[Raw moments]
If $X\sim \text{BN}(\boldsymbol{\theta})$ then 
{\scalefont{0.89}
\begin{align*}
\mathbb{E}(X^n) =
\begin{cases}
\sigma^n 2^{n-2\over2}\, 
{\Gamma({n+1\over 2})\over\sqrt{\pi}}\,
\big[ \, 
_1\!F_1\big(-{n\over 2},{1\over 2};-{\{\mu+\alpha\sigma\}^2\over 2\sigma^2}\big)
+ 
_1\!\!F_1\big(-{n\over 2},{1\over 2};-{\{\mu-\alpha\sigma\}^2\over 2\sigma^2}\big)
\big], & \hspace{-0.3cm} \ n \ \text{even},
\\[0,35cm]
\sigma^{n-1}2^{n-1\over 2}\,  
{\Gamma({n\over 2}+1)\over\sqrt{\pi}}\,
\big[
(\mu+\alpha\sigma)
_1\!F_1\big({1-n\over 2},{3\over 2};-{\{\mu+\alpha\sigma\}^2\over 2\sigma^2}\big)
+
(\mu-\alpha\sigma)
_1\!F_1\big({1-n\over 2},{3\over 2};-{\{\mu-\alpha\sigma\}^2\over 2\sigma^2}\big)
\big], & \hspace{-0.3cm} \ n \ \text{odd},
\end{cases}
\end{align*}
}
where
$
_1\!F_1(a,b;x)=[\Gamma(b)/\Gamma(a)]
\sum_{k=0}^{\infty} 
[\Gamma(a+k)/\Gamma(b+k) ] (x^k / {k!})
$
is the Kummer's confluent hypergeometric function; see \cite{{Win:14}}.
\end{proposition}
\begin{proof}
	
By Proposition \ref{Stochastic representation} we have
\begin{align*}
\mathbb{E}(X^{n})
=
{1\over 2}\, \big[\mathbb{E}_{\Phi_{\mu+\alpha\sigma, \sigma^2 }}(X^n)+\mathbb{E}_{\Phi_{\mu-\alpha\sigma, \sigma^2 }}(X^n)\big],
\end{align*}
where $\mathbb{E}_{\Phi_{\mu+\alpha\sigma, \sigma^2 }}(\cdot)$ denotes the expectation with respect to distribution function $\Phi_{\mu+\alpha\sigma, \sigma^2 }$.

By combining the above equality with the following known identity \citep{Win:14}, for $Y\sim N(\mu,\sigma^2)$,
\begin{align*}
\mathbb{E}(Y^n) =
\begin{cases}
\sigma^n 2^{n/2}\, 
{\Gamma({n+1\over 2})\over\sqrt{\pi}}\,
_1\!F_1\big(-{n\over 2},{1\over 2};-{\mu^2\over 2\sigma^2}\big), &  n \ \text{even},
\\[0,35cm]
\mu\sigma^{n-1} 2^{(n+1)/2}\,  
{\Gamma({n\over 2}+1)\over\sqrt{\pi}}\,
_1\!F_1\big({1-n\over 2},{3\over 2};-{\mu^2\over 2\sigma^2}\big), &  n \ \text{odd},
\end{cases}
\end{align*}
the proof follows.
\end{proof}

\begin{proposition}[Standardized moments]
If $X\sim \text{BN}(\boldsymbol{\theta})$ then 
\begin{align*}
\mathbb{E}\biggl[\biggl({X-\mu\over\sqrt{{\rm Var}(X)}}\biggr)^n\,\biggr]=
\begin{cases}
{1\over (1+\alpha^2)^{n/2}}
{\sum_{\substack{0\leq k\leq n\\ k\ {\rm even}}}} \binom{n}{k} \alpha^{n-k}
2^{-{k\over 2}}\, {k!\over (k/2)!}, &  n \ \text{even},
\\[0,35cm]
0, & n \ \text{odd}.
\end{cases}
\end{align*}
\end{proposition}
\begin{proof}
By using Proposition \ref{Stochastic representation} and that ${\rm Var}(X)=\sigma^2(1+\alpha^2)$, we get
\begin{align*}
\mathbb{E}\biggl[\biggl({X-\mu\over\sqrt{{\rm Var}(X)}}\biggr)^n\,\biggr]
=
{1\over 2(1+\alpha^2)^{n/2}}\, 
\left\{
\mathbb{E}_{\Phi_{\mu+\alpha\sigma, \sigma^2 }}\biggl[\biggl({X-\mu\over\sigma}\biggr)^n\,\biggr]
+
\mathbb{E}_{\Phi_{\mu-\alpha\sigma, \sigma^2 }}\biggl[\biggl({X-\mu\over\sigma}\biggr)^n\,\biggr]
\right\},
\end{align*}
where $\mathbb{E}_{\Phi_{\mu+\alpha\sigma, \sigma^2 }}(\cdot)$ denotes the expectation with respect to distribution function $\Phi_{\mu+\alpha\sigma, \sigma^2 }$.
Taking the change of variable $z=(x-\mu)/\sigma$, ${\rm d}z={\rm d}x/\sigma$, and a binomial expansion, we have
\begin{align} \label{id-1}
\mathbb{E}\biggl[\biggl({X-\mu\over\sqrt{{\rm Var}(X)}}\biggr)^n\,\biggr]
&=
{1\over 2 (1+\alpha^2)^{n/2}}\, 
\left\{
\mathbb{E}_{\Phi}\big[(Z+\alpha)^n\,\big]
+
\mathbb{E}_{\Phi}\big[(Z-\alpha)^n\,\big]
\right\} \nonumber
\\[0,2cm]
&=
{1\over 2(1+\alpha^2)^{n/2}}\, 
\sum_{k=0}^{n}\binom{n}{k} \big[1+(-1)^{n-k}\big] \alpha^{n-k}
\mathbb{E}_{\Phi}(Z^k).
\end{align}
A simple observation shows that, when $n$ is even,
\begin{align} \label{id-2}
{1\over 2(1+\alpha^2)^{n/2}}\, 
\sum_{k=0}^{n}\binom{n}{k} \big[1+(-1)^{n-k}\big] \alpha^{n-k}
\mathbb{E}_{\Phi}(Z^k)
=
{1\over (1+\alpha^2)^{n/2}}
{\sum_{\substack{0\leq k\leq n\\ k\ {\rm even}}}} \binom{n}{k} \alpha^{n-k} \mathbb{E}_{\Phi}(Z^k),
\end{align}
and, when $n$ is odd,
\begin{align}\label{id-3}
{1\over 2(1+\alpha^2)^{n/2}}\, 
\sum_{k=0}^{n}\binom{n}{k} \big[1+(-1)^{n-k}\big] \alpha^{n-k}
\mathbb{E}_{\Phi}(Z^k)
=
{1\over (1+\alpha^2)^{n/2}}
{\sum_{\substack{0\leq k\leq n\\ k\ {\rm odd}}}} \binom{n}{k} \alpha^{n-k} \mathbb{E}_{\Phi}(Z^k).
\end{align}
Finally, by combining the known identities, $\mathbb{E}_{\Phi}(Z^k)=0$ for $k$ odd, and 
\begin{align*}
\mathbb{E}_{\Phi}(Z^k)=2^{-{k\over 2}}\, {k!\over (k/2)!},
\end{align*}
for $k$ even, with \eqref{id-1}, \eqref{id-2} and \eqref{id-3}, the proof follows.
\end{proof}

\section{Identifiability of the BN distribution}\label{Identifiability}

As a consequence of Proposition \ref{Stochastic representation} we know that the BN PDF $f(x;\boldsymbol{\theta})$ in \eqref{density-BN}, with parameter vector $\boldsymbol{\theta} =(\mu,\sigma,\alpha)$, can be written as a finite mixture of two normal distributions with different location parameters, i.e.
\begin{align}\label{eq-eq-dens}
f(x;\boldsymbol{\theta})
=
{1\over 2}\, 
\big[\phi_{\mu+\alpha\sigma, \sigma^2 }(x)
+
\phi_{\mu-\alpha\sigma, \sigma^2}(x)\big].
\end{align}

\smallskip 
Let $\mathcal{N}$ be the family of normal distributions, as follows:
\begin{align*}
\mathcal{N}=\biggl\{F:F(x;\mu,\sigma)=\int_{-\infty}^{x}\phi_{\mu, \sigma^2 }(y)\, {\rm d}y, \ \mu\in\mathbb{R}, \sigma>0,\  x\in\mathbb{R} \biggl\}.
\end{align*}

Write $\mathcal{H}_{\mathcal{N}}$ the class of all finite mixtures of $\mathcal{N}$. It is well-known that the class $\mathcal{H}_{\mathcal{N}}$ is identifiable \citep{Teicher:63}.

The following result proves the identifiability of BN distribution.
\begin{proposition}
	The mapping $\boldsymbol{\theta} \longmapsto f(x;\boldsymbol{\theta})$, for all $x\in\mathbb{R}$, is one-to-one.
\end{proposition}
\begin{proof}
Let us suppose that $f(x;\boldsymbol{\theta})=f(x;\boldsymbol{\theta}')$ for all $x\in\mathbb{R}$. In other words, by \eqref{eq-eq-dens},
\begin{align*}
{1\over 2}\, 
\big[\phi_{\mu+\alpha\sigma, \sigma^2 }(x)
+
\phi_{\mu-\alpha\sigma, \sigma^2}(x)\big]
=
{1\over 2}\, 
\big[\phi_{\mu'+\alpha'\sigma', \sigma'^2 }(x)
+
\phi_{\mu'-\alpha'\sigma', \sigma'^2}(x)\big].
\end{align*}
Since $\mathcal{H}_{\mathcal{N}}$ is identifiable, we have
$\mu\pm\alpha\sigma=\mu'\pm\alpha'\sigma'$ and $\sigma^2=\sigma'^2$.  From where, immediately follows that $\mu=\mu'$, $\sigma=\sigma'$ and $\alpha=\alpha'$. Therefore, $\boldsymbol{\theta}=\boldsymbol{\theta}'$, and the identifiability of distribution follows.
\end{proof}

\section{Asymptotic properties}\label{sec:asymptot}
Let $X$ be a random variable with BN distribution $f(x;\boldsymbol{\theta})$ that depends on a parameter vector $\boldsymbol{\theta}= (\mu,\sigma,\alpha)$, for $\boldsymbol{\theta}$ in an open subset of $\mathbb{R}^3$, where
distinct values of $\boldsymbol{\theta}$ yield distinct distributions for $X$ (see Section \ref{Identifiability}).
Let $\boldsymbol{X}=(X_1,\ldots, X_n)$ be a random sample of $X$. The log-likelihood function for $\boldsymbol{\theta}$ is given by 
\begin{align*}
l(\boldsymbol{\theta};\boldsymbol{X})
=
{\rm const}-n\log(\sigma)-{n\alpha^2\over 2}
-{1\over 2}\, \sum_{i=1}^{n} \biggl({X_i-\mu\over\sigma}\biggr)^2
+
\sum_{i=1}^{n}\log\cosh\biggl[\alpha\biggl({X_i-\mu\over \sigma}\biggr)\biggr].
\end{align*}
A simple computation shows that
\begin{align}
&
{\partial l(\boldsymbol{\theta};\boldsymbol{X})\over \partial\mu}
=
{n\over \sigma}\, \biggl({\overline{X}-\mu\over\sigma}\biggr)
-
{\alpha\over\sigma}\sum_{i=1}^n {\rm tanh}\biggl[\alpha\biggl({ X_i-\mu\over \sigma}\biggl)\biggr], \label{score}
\\[0,2cm]
&
{\partial l(\boldsymbol{\theta};\boldsymbol{X})\over \partial\sigma}
=
-{n\over\sigma}
+
{1\over \sigma}\, \sum_{i=1}^{n} \biggl({X_i-\mu\over \sigma}\biggl)^2
-
{\alpha\over\sigma}\sum_{i=1}^n \biggl({X_i-\mu\over \sigma}\biggl) {\rm tanh}\biggl[\alpha\biggl({X_i-\mu\over \sigma}\bigg)\biggr], \label{score-1}
\\[0,2cm]
&
{\partial l(\boldsymbol{\theta};\boldsymbol{X})\over \partial\alpha}
=
-\alpha n+ \sum_{i=1}^n \biggl({X_i-\mu\over\sigma}\biggr) {\rm tanh}\biggl[\alpha\biggl({ X_i-\mu\over \sigma}\biggr)\biggr]. \label{score-2}
\end{align}
The maximum log-likelihood equations for the estimators $\widehat{\mu}$, $\widehat{\sigma}$, $\widehat{\alpha}$  are as follows:
\begin{align*}
& 
\widehat{\mu}
=
\overline{X}
-
{\widehat{\alpha}\over n}
\sum_{i=1}^n {\rm tanh}\biggl[\widehat{\alpha} \biggl({ X_i-\widehat{\mu}\over \widehat{\sigma}}\biggr)\biggr],
\\[0,2cm]
&
\widehat{\sigma}^2
=
{1\over (1+\widehat{\alpha}^2)n}\, 
\sum_{i=1}^{n} (X_i-\widehat{\mu})^2,
\\[0,2cm]
&\widehat{\alpha} =
{1\over n}\,
\sum_{i=1}^n \biggl({X_i-\widehat{\mu}\over\widehat{\sigma}}\biggr) {\rm tanh}\biggl[\widehat{\alpha}\biggl({ X_i-\widehat{\mu}\over \widehat{\sigma}}\biggr)\biggr].
\end{align*}

In the following two propositions we study the existence 
of the ML estimates when the other parameters are known.
\begin{proposition}\label{prop-existence-roots}
If the parameters $\sigma$ and $\alpha$ are known, then the equation \eqref{score} has at least one root on the interval $(-\infty, +\infty)$.
\end{proposition}
\begin{proof}
One can readily verify that $\lim_{\mu\to\mp\infty} {\partial l(\boldsymbol{\theta};\boldsymbol{X})\over \partial\mu} =\pm\infty$. So, by Intermediate value theorem, there exists at least one solution on the interval $(-\infty, +\infty)$.
\end{proof}

\begin{proposition}
If the parameters $\mu$ and $\sigma$ are known, then the equation \eqref{score-2} has at least one root on the interval $(-\infty, +\infty)$.
\end{proposition}
\begin{proof}
Since $\lim_{\alpha\to\mp \infty} {\partial l(\boldsymbol{\theta};\boldsymbol{X})\over \partial\alpha} =\pm\infty$, the proof follows the same reasoning as Proposition \ref{prop-existence-roots}.
\end{proof}

Now, we calculate the expectation of score defined by \eqref{score}, \eqref{score-1} and\eqref{score-2} when $n=1$. Indeed,
by using the partial derivatives in \eqref{score}-\eqref{score-2}, with $n=1$, and the fact that $x\longmapsto x\cosh(\alpha x)$ and $x\longmapsto\sinh(\alpha x)$ are odd functions, we obtain
\begin{align*}
\mathbb{E}\biggl[{\partial \log f(X;\boldsymbol{\theta})\over \partial \mu }\biggr]
&=
{n\over \sigma}\, 
\mathbb{E}\biggl({X-\mu\over\sigma}\biggr)
-
{\alpha\over\sigma}\, 
\mathbb{E}\biggl\{{\rm tanh}\biggl[\alpha\biggl({ X-\mu\over \sigma}\biggr)\biggr]\biggr\}
\\[0,2cm]
&=
\exp\biggl(-{\alpha^2\over 2}\biggr)
\left\{
{n\over \sigma}\,
\mathbb{E}_{\Phi}
\big[Z\cosh(\alpha Z)\big]
-
{\alpha\over\sigma}\, 
\mathbb{E}_{\Phi}
\big[\sinh(\alpha Z)\big]
\right\}
=0,
\end{align*}
where in the second line the following change of variables $z=(x-\mu)/\sigma$, ${\rm d}z={\rm d}x/\sigma$, was taken.

Analogously, since $
\mathbb{E}_{\Phi}\big[Z^2 \cosh(\alpha Z) \big]=(\alpha^2+1)\exp\big({\alpha^2/ 2}\big)$ and $\mathbb{E}_{\Phi}
\big[Z\sinh(\alpha Z)\big]=\alpha\exp\big({\alpha^2/ 2}\big)$, we get
\begin{align*}
\mathbb{E}\biggl[{\partial \log f(X;\boldsymbol{\theta})\over \partial \sigma }\biggr]
&=
-{1\over\sigma}
+
{1\over \sigma}\, \mathbb{E}\biggl[\biggl({X-\mu\over \sigma}\biggl)^2\, \biggr]
-
{\alpha\over\sigma}\, \mathbb{E}\biggl\{ \biggl({X-\mu\over \sigma}\biggl) {\rm tanh}\biggl[\alpha\biggl({X-\mu\over \sigma}\bigg)\biggr]\biggr\}
\\[0,2cm]
&=
-{1\over\sigma}
+
{1\over \sigma}\, \exp\biggl(-{\alpha^2\over 2}\biggr)
\mathbb{E}_{\Phi}\big[Z^2 \cosh(\alpha Z) \big]
-
{\alpha\over\sigma}\, 
\exp\biggl(-{\alpha^2\over 2}\biggr)
\mathbb{E}_{\Phi}
\big[Z\sinh(\alpha Z)\big]
=0
\end{align*}
and
\begin{align}\label{exp=10}
\mathbb{E}\biggl[{\partial \log f(X;\boldsymbol{\theta})\over \partial \alpha }\biggr]
&=
\mathbb{E}
\biggl\{\biggl({ X-\mu\over \sigma}\biggr)
\tanh\biggl[\alpha\biggl({ X-\mu\over \sigma}\biggr)\biggr]\biggr\}
-\alpha \nonumber
\\[0,2cm]
&=
\exp\biggl(-{\alpha^2\over 2}\biggr)
\mathbb{E}_{\Phi}
\big[Z\sinh(\alpha Z)\big]
-\alpha
\ = \ 
0.
\end{align}

\subsection{Consistence of the MLE $\widehat{\alpha}$}
For the sake of simplicity of presentation, from now on we will assume that {\it $\mu$ and $\sigma$ are known parameters and $\alpha$ is unknown}. We are interested in knowing the large sample properties of  MLE $\widehat{\alpha}$ of the parameter $\alpha$ that generates uni- or bimodality in the BN distribution. We emphasize that similar results can be studied for $\mu$ and $\sigma$ when the other parameters are known.

Since 
\begin{align*}
{\partial^2 f(x;\boldsymbol{\theta})\over \partial \alpha^2}
=
\biggl[\alpha^2+
\biggl({x-\mu\over \sigma}\biggr)^2-1\biggl]
f(x;\boldsymbol{\theta})
-
2\alpha \biggl({x-\mu\over \sigma}\biggr)
\tanh\biggl[\alpha\biggl({x-\mu\over \sigma}\biggr)\biggr]
f(x;\boldsymbol{\theta}) 
\end{align*}
and, since $
\mathbb{E}_{\Phi}\big[Z^2 \cosh(\alpha Z) \big]=(\alpha^2+1)\exp\big({\alpha^2/ 2}\big)$ and $\mathbb{E}_{\Phi}
\big[Z\sinh(\alpha Z)\big]=\alpha\exp\big({\alpha^2/ 2}\big)$,
for $Z\sim N(0,1)$, we have
\begin{align}\label{int-cond}
\int_{-\infty}^{+\infty} {\partial^2 f(x;\boldsymbol{\theta})\over \partial \alpha^2}\, {\rm d}x
&=
\mathbb{E}\biggl[\alpha^2+
\biggl({ X-\mu\over \sigma}\biggr)^2-1\biggl]
-
2\alpha
\mathbb{E}
\biggl\{\biggl({ X-\mu\over \sigma}\biggr)
\tanh\biggl[\alpha\biggl({ X-\mu\over \sigma}\biggr)\biggr]\biggr\}
\nonumber
\\[0,25cm]
&=
\alpha^2+
\exp\biggl(-{\alpha^2\over 2}\biggr)
\mathbb{E}_{\Phi}\big[Z^2 \cosh(\alpha Z)\big]
-1
-
2\alpha 
\exp\biggl(-{\alpha^2\over 2}\biggr)
\mathbb{E}_{\Phi}
\big[Z\sinh(\alpha Z)\big]
\nonumber
\\[0,2cm]
&=
0,
\end{align}
where in the second line the following change of variables $z=(x-\mu)/\sigma$, ${\rm d}z={\rm d}x/\sigma$, was taken.

On the other hand, 
\begin{align}\label{id-sech}
{\partial^2 \log f(x;\boldsymbol{\theta})\over \partial \alpha^2 }
=
\biggl({x-\mu\over \sigma}\biggr)^2
{\rm sech}^2\biggl[\alpha\biggl({x-\mu\over \sigma}\biggr)\biggr] -1 .
\end{align}

Then, by \eqref{int-cond} and \eqref{id-sech},
the Fisher information may also be written as
\begin{align}\label{IF}
\mathcal{I}(\alpha)
=
\mathbb{E}\biggl[{\partial \log f(X;\boldsymbol{\theta})\over \partial \alpha }\biggr]^2 \nonumber
&=
-\mathbb{E}\biggl[{\partial^2 \log f(X;\boldsymbol{\theta})\over \partial \alpha^2 }\biggr]
+
\int_{-\infty}^{+\infty} {\partial^2 f(x;\boldsymbol{\theta})\over \partial \alpha^2}\, {\rm d}x \nonumber
\\[0,3cm]
&{=}
1- \mathbb{E}\biggl\{\biggl({X-\mu\over \sigma}\biggr)^2
{\rm sech}^2\biggl[\alpha\biggl({X-\mu\over \sigma}\biggr)\biggr]\biggl\} \nonumber
\\[0,3cm]
&{=}
1-\exp\biggl(-{\alpha^2\over 2}\biggr)\mathbb{E}_{\Phi}
\big[Z^2{\rm sech}(\alpha Z)\big].
\end{align}

\begin{theorem}[Consistence]\label{consistence}
Let us suppose that $\mu$ and $\sigma$ are known parameters and $\alpha$ unknown.
Let $\Theta=\{\alpha\in\mathbb{R}:\vert \alpha\vert >0\}$ be the parameter space. Then, with probability approaching 1, as $n\to +\infty$, the log-likelihood equation ${\partial l(\boldsymbol{\theta};\boldsymbol{X})/ \partial \alpha }= 0$ has a consistent solution, denoted by $\widehat{\alpha}$.
\end{theorem}
\begin{proof}
Since ${\partial \log f(x;\boldsymbol{\theta})/ \partial \alpha }$, ${\partial^2 \log f(x;\boldsymbol{\theta})/ \partial \alpha^2 }$, ${\partial^3 \log f(x;\boldsymbol{\theta})/ \partial \alpha^3 }$ exist for all $\alpha\in\Theta$ and every $x$, by \cite{cra:46} it is sufficient to prove that:
\begin{enumerate}
	\item $\mathbb{E}[{\partial \log f(X;\boldsymbol{\theta})/ \partial \alpha }]=0$ for all $\alpha\in\Theta$;
	\item $-\infty<\mathbb{E}[{\partial^2 \log f(X;\boldsymbol{\theta})/ \partial \alpha^2 }]<0$ for all $\alpha\in\Theta$;
	\item There exists a function $H(x)$ such that for all $\alpha\in\Theta$,
	\begin{align*}
	\biggl\vert {\partial^3 \log f(x;\boldsymbol{\theta})\over \partial \alpha^3 }\biggl\vert <H(x) 
	\quad \text{and}\quad 
	\mathbb{E}[H(X)]<\infty.
	\end{align*}
\end{enumerate}

In what follows we show the validity of Items 1, 2 and 3 above.

By \eqref{exp=10}, the statement of Item 1 follows.

In order to verify the second item,
note that, 
$\exp\bigl(-{\alpha^2/ 2}\bigr)\mathbb{E}_{\Phi}
\big[Z^2{\rm sech}(\alpha Z)\big]\leq \mathbb{E}_{\Phi}(Z^2)=1$ for all $\alpha\in\Theta$.
Moreover, the two sides are equal if and only if $\alpha=0$. Since $\alpha\in\Theta$ (that is, $\alpha\neq 0$), it follows that
$\exp\big(-{\alpha^2/ 2}\big)\mathbb{E}_{\Phi}
\big[Z^2{\rm sech}(\alpha Z)\big]<1$. Hence,
%
\begin{align}\label{exp-finite}
-1\leq 
\mathbb{E}\biggl[{\partial^2 \log f(X;\boldsymbol{\theta})\over \partial \alpha^2 }\biggr]
\stackrel{\eqref{IF}}{=}
\exp\biggl(-{\alpha^2\over 2}\biggr)\mathbb{E}_{\Phi}
\big[Z^2{\rm sech}(\alpha Z)\big]-1
<0.
\end{align}
Then Item 2 is valid.

Finally, since $\vert{\rm sech}^2(\alpha x)\vert \leq 1$ and $\vert{\rm tanh}(\alpha x)\vert\leq 1$, 
\begin{align}\label{3condition}
\biggl\vert{\partial^3 \log f(x;\boldsymbol{\theta})\over \partial \alpha^3 }\biggr\vert
&=
\Biggl\vert 
2\biggl({X-\mu\over \sigma}\biggr)^3{\rm sech}^2\biggl[\alpha\biggl({X-\mu\over \sigma}\biggr)\biggr] {\rm tanh}\biggl[\alpha\biggl({X-\mu\over \sigma}\biggr)\biggr]\Biggr\vert
\nonumber
\\[0,3cm]
&\leq 
2\biggl\vert \biggl({X-\mu\over \sigma}\biggr)\biggr\vert^3=H(x),
\quad \text{with} \quad \mathbb{E}[H(X)]<\infty.
\end{align}

Thus we have complete the proof.
\end{proof}

The following simple result further supports the intuitive appeal of the MLE \citep{baj:71}.
\begin{proposition}
Under hypothesis of Theorem \ref{consistence} it holds:
\begin{align*}
\lim_{n\to+\infty} 
\int_{\mathbb{R}^n} 
\mathbbm{1}_{\displaystyle\{\boldsymbol{y}\in \mathbb{R}^n: \exp[l(\boldsymbol{\theta}';\boldsymbol{y})]>\exp[l(\boldsymbol{\theta};\boldsymbol{y})]\}} (\boldsymbol{x}) \exp[l(\boldsymbol{\theta}';\boldsymbol{x})]\, 
{\rm d} \boldsymbol{x}
=1,
\end{align*} 
for any $\boldsymbol{\theta}=(\mu,\sigma,\alpha)$, $\boldsymbol{\theta}'=(\mu,\sigma,\alpha')\in\Theta$ with $\alpha\neq\alpha'$. Here, $\mathbbm{1}_A(x)$ is the indicator function of a set $A$ having the value 1 for all $x$ in $A$ and the value 0 for all $x$ not in $A$.
\end{proposition}
\begin{proof}
Since $\boldsymbol{X}=(X_1,\ldots, X_n)$ is a random sample of $X\sim BN(\boldsymbol{\theta})$,
$X_1, \ldots ,X_n$ are independent and identically distributed with density $f (x;\boldsymbol{\theta}), \boldsymbol{\theta}\in\Theta$; and since the BN distribution is identifiable (see Section \ref{Identifiability}), by \cite{baj:71} the proof follows.
\end{proof}

\subsection{Central limit theorem for the MLE $\widehat{\alpha}$}

In this section we state a Central limit theorem (CLT) for the MLE $\widehat{\alpha}$, which is important for studying confidence intervals and hypothesis tests, for example.

Note that, under hypothesis of Theorem \ref{consistence}, the following conditions are satisfied:

\begin{enumerate}
\item[(A.1)] The mapping $\alpha \longmapsto f(x; \boldsymbol{\theta})$ is three times continuously differentiable on $\Theta$, $\forall x\in\mathbb{R}$;
\item[(A.2)] By \eqref{exp=10}, $\int_{-\infty}^{+\infty} {\partial f(x;\boldsymbol{\theta})\over \partial \alpha}\, {\rm d}x=\mathbb{E}\big[{\partial \log f(X;\boldsymbol{\theta})\over \partial \alpha }\big]=0$ and, by \eqref{int-cond},
$
\int_{-\infty}^{+\infty} {\partial^2 f(x;\boldsymbol{\theta})\over \partial \alpha^2}\, {\rm d}x=0;
$
\item[(A.3)]  By \eqref{IF} and \eqref{exp-finite},
$
0<
\mathcal{I}(\alpha)
=
1- \mathbb{E}\big[X^2{\rm sech}^2(\alpha X)\big]
\leq 1
$,
$\forall \alpha\in\Theta$;
\item[(A.4)] 
By \eqref{3condition},
there exists a function $H(x)$ such that for all $\alpha\in\Theta$,
\begin{align*}
\biggl\vert {\partial^3 \log f(x;\boldsymbol{\theta})\over \partial \alpha^3 }\biggl\vert <H(x) 
\quad \text{and}\quad 
\mathbb{E}[H(X)]<\infty;
\end{align*}
\item[(A.5)] By Theorem \ref{consistence}, the log-likelihood equation ${\partial l(\boldsymbol{\theta};\boldsymbol{X})/ \partial \alpha }= 0$ has a consistent solution $\widehat{\alpha}$.
\end{enumerate}
Since conditions (A.1)-(A.5) are satisfied, by \cite{cra:46} the following result follows:

\begin{theorem}[CLT for the MLE]
Under hypothesis of Theorem \ref{consistence},
it holds that, $\sqrt{n}(\widehat{\alpha}- \alpha)$ converges in distribution to $N(0, 1/\mathcal{I}(\alpha))$ as $n\to +\infty$.
\end{theorem}

%

\section{The bivariate BN distribution}\label{sec:bivariate}
We said that a real random vector $\boldsymbol{X}=(X_1,X_2)$ has bivariate BN (BBN) distribution with parameter vector parameter $\boldsymbol{\psi }=(\mu_1,\mu_2,\sigma_1,\sigma_2,\alpha)$, $\mu_i\in\mathbb{R}$, $\sigma_i>0$, $\alpha\in \mathbb{R}$, denoted by $\boldsymbol{X}\sim \text{BBN}(\boldsymbol{\psi })$, if its PDF is given by, for each  ${\bf x}=(x_1, x_2)\in\mathbb{R}^2$,
\begin{eqnarray*}
f({\bf x};\boldsymbol{\psi })
=
{\exp[{\alpha^2(\rho^2-2)/ 2}]\over\sigma_1\sigma_2}\,
\phi\biggl({x_1-\mu_1\over \sigma_1},{x_2-\mu_2\over \sigma_2};\rho\biggr)
\cosh
\biggl[ \alpha\biggl({x_1-\mu_1\over \sigma_1}\biggr)+ \alpha(1-\rho)\biggl({x_2-\mu_2\over \sigma_2}\biggr) \biggr],
\end{eqnarray*}
where $\rho\in(-1,1)$  and
\begin{align*}
\phi({\bf z};\rho)=\dfrac{1}{2\pi\sqrt{1-\rho^2}}\,
\exp\biggl[-{1\over 2 (1-\rho^2)}\big( 
z_1^2 - 2\rho z_1z_2  +z_2^2
\big)\biggr],
\quad {\bf z}=(z_1,z_2),
\end{align*}
is the  PDF of the standard bivariate normal distribution with correlation coefficient $\rho$.

A simple algebraic manipulation shows that
\begin{align}\label{ident-binormal}
{1\over\sigma_1\sigma_2}\,
\phi\biggl({x_1-\mu_1\over \sigma_1},{x_2-\mu_2\over \sigma_2};\rho\biggr)
=
\phi_{\widetilde{\mu}_1, \widetilde{\sigma}_1^2}(x_1) \,
\phi_{\mu_2,\sigma_2^2}(x_2),
\end{align}
where
\begin{align*}
\widetilde{\mu}_1=\mu_1+\rho\sigma_1\biggl({x_2-\mu_2\over\sigma_2}\biggr)
\quad \text{and} \quad 
\widetilde{\sigma}_1^2=\sigma_1^2(1-\rho^2).
\end{align*}
Consequently
\begin{align*}
	\int_{-\infty}^{\infty}f({\bf x};\boldsymbol{\psi})\, {\rm d}x_1
	=
	f(x_2;\boldsymbol{\theta}_{2}) 
	\quad \text{and}\quad
	\int_{-\infty}^{\infty}f({\bf x};\boldsymbol{\psi})\, {\rm d}x_2
=
f(x_1;\boldsymbol{\theta}_{1}),
\end{align*}
where $f(x_i;\boldsymbol{\theta}_{i})$ is the PDF of the BN distribution \eqref{density-BN} with parameter vector $\boldsymbol{\theta}_{i}=(\mu_i,\sigma_i,\alpha)$, $i=1,2$. In other words, If $\boldsymbol{X}=(X_1,X_2)\sim \text{BBN}(\boldsymbol{\psi })$ then $X_1\sim {\rm BN}(\boldsymbol{\theta}_1)$ and $X_2\sim {\rm BN}(\boldsymbol{\theta}_2)$.

By using \eqref{ident-binormal}, a laborious algebraic calculation gives the following 
\begin{align*}
\mathbb{E}(X_1\vert X_2=x_2)
=
\mu_1
+
\rho \sigma_1 \biggl({ x_2-\mu_2\over \sigma_2}\biggr)
+
(1-\rho^2)\sigma_1
{\rm tanh}\biggl[\alpha\biggl({ x_2-\mu_2\over \sigma_2}\biggr)\biggr].
\end{align*}
That is, 
\begin{align*}
\mathbb{E}(X_1\vert X_2)
=
\mu_1
+
\rho \sigma_1 \biggl({ X_2-\mu_2\over \sigma_2}\biggr)
+
(1-\rho^2)\sigma_1
{\rm tanh}\biggl[\alpha\biggl({ X_2-\mu_2\over \sigma_2}\biggr)\biggr] \quad \text{a.s.}
\end{align*}
In consequence,
\begin{align*}
\mathbb{E}(X_1X_2)&=\mathbb{E}[X_2\mathbb{E}(X_1\vert X_2)]
\\[0,2cm]
&=
\mu_1\mathbb{E}(X_2)
+
\rho \sigma_1 \mathbb{E}\biggl[X_2\biggl({ X_2-\mu_2\over \sigma_2}\biggr)\biggr]
+
(1-\rho^2) \sigma_1
\mathbb{E}\biggl\{X_2
{\rm tanh}\biggl[\alpha\biggl({ X_2-\mu_2\over \sigma_2}\biggr)\biggr]
\biggr\}.
\end{align*}
Since $X_2\sim {\rm BN}(\boldsymbol{\theta}_2)$ we get
\begin{align*}
\mathbb{E}(X_1X_2)=
\mu_1\mu_2+\rho\sigma_1\sigma_2(1+\alpha^2)+(1-\rho^2)\sigma_1\sigma_2\alpha.
\end{align*}
Hence, since $\mathbb{E}(X_i)=\mu_i$ and ${\rm Var}(X_i)=\sigma^2_i(1+\alpha^2)$ (see properties P.7 and P.8 in Section \ref{sec:pre:prop}),
\begin{align}
{\rm Cov}(X_1,X_2)
&=
\sigma_1\sigma_2[
\rho(1+\alpha^2)+(1-\rho^2)\alpha]; \label{cov}
\\[0,2cm]
{\rm \rho}(X_1,X_2)&=
\dfrac{
	\rho(1+\alpha^2)+(1-\rho^2)\alpha}{(1+\alpha^2)}. \nonumber
\end{align}
The covariance matrix is given by
\begin{align*}
\Sigma=
\begin{bmatrix}
\sigma^2_1(1+\alpha^2) & \sigma_1\sigma_2[
\rho(1+\alpha^2)+(1-\rho^2)\alpha] \\[0,3cm]
\sigma_1\sigma_2[
\rho(1+\alpha^2)+(1-\rho^2)\alpha]   &  \sigma^2_2(1+\alpha^2) 
\end{bmatrix}.
\end{align*}

Some immediate observations are as follows:
\begin{itemize}
	\item
When $\alpha=0$ we have the following known facts corresponding to    bivariate normal distribution: ${\rm Cov}(X_1,X_2)=\rho\sigma_1\sigma_2$ and ${\rm \rho}(X_1,X_2)=\rho$.

	\item
When $\rho=0$ we have ${\rm Cov}(X_1,X_2)=\sigma_1\sigma_2\alpha$ and ${\rm \rho}(X_1,X_2)=\alpha/(1+\alpha^2)$.

	\item
When $\rho=\alpha=0$, $X_1$ and $X_2$ are independent.
\end{itemize}

\section{Stationarity and ergodicity}\label{sec:stat:ergo}

\subsection{Non-stationarity of the BN random process}

\begin{definition}
A process $X_t$ is strict-sense stationary (SSS) if its finite-dimensional distributions at times $t_1<\cdots<t_n$, $\forall n\in\mathbb{N}$, are the same after any
time interval of length time interval of length $t_0$. In other words, for each $n\in\mathbb{N}$ and $t_1<\cdots<t_n$ and $(x_1,\dots,x_n)\in\mathbb{R}^n$ we have
\begin{align*}
\mathbb{P}(X_{t_1+t_0}\leq x_1,\dots,X_{t_n+t_0}\leq x_n)
=
\mathbb{P}(X_{t_1}\leq x_1,\dots,X_{t_n}\leq x_n),
\end{align*}
for any time $t_0$.
\end{definition}

We said that a process $X_t$ is a  BN random process if $X_t\sim \text{BN}(\boldsymbol{\theta}_t)$, where $\boldsymbol{\theta}_t=(\mu_t,\sigma_t,\alpha)$, $\mu_t\in\mathbb{R}$, $\sigma_t>0$ and $\alpha\in \mathbb{R}$. 

\begin{proposition}\label{not-SSS}
The BN random process is not SSS when $\mu_t$ and $\sigma_t$ are not independent of time.
\end{proposition}
\begin{proof}
If a random process is SSS, then all expected values of functions of the random process, must also be stationary. Since $\mathbb{E}(X_t)=\mu_t$ and ${\rm Var}(X_t)=\sigma^2_t(1+\alpha^2)$ (see properties P.7 and P.8 in Section \ref{sec:pre:prop}) change in time, we have that the PDF change with time. Then the not stationarity of random process follows.
%
\end{proof}

\begin{definition}
A process $X_t$ is weak-sense stationary (WSS) if:
\begin{itemize}
\item $\mathbb{E}(X_t)=\mu$  is independent of time;
\item $\mathbb{E}(X_t^2)<\infty$;
\item $C_X(t,s)={\rm Cov}(X_t,X_s)$ only depends on the distance 
between the times considered.
\end{itemize}
\end{definition}

If $X_t$ is a BN random process,  it is known that $\mathbb{E}(X_t)=\mu_t$, $\mathbb{E}(X_t^2)=\sigma_t^2(1+\alpha^2)+\mu_t^2$ (see Section \ref{sec:pre:prop}) and that $C_X(t,s)\stackrel{\eqref{cov}}{=}\sigma_t\sigma_s[
\rho(1+\alpha^2)+(1-\rho^2)\alpha]$. Then the next result follows.

\begin{proposition}\label{not-WSS}
	The BN random process is not WSS when $\mu_t$ and $\sigma_t$ are not independent of time.
\end{proposition}

\begin{remark}
	In the case that $\mu_t$ and $\sigma_t$ [or $\rho=\alpha=0$] are independent of time, it is clear that the BN process is SSS and WSS.
\end{remark}

\subsection{Mean, variance and covariance ergodicity of the BN random process}
In many real-life situations, it is not always possible to have many realizations of the random process available to estimate a population parameter (for example, the mean, variance and covariance function of process), as is customary in classical estimation, but rather a single one. In this case,  in order to study the process, we calculate the temporal characteristic in order to study the process characteristic.

\begin{definition}
Let $X_t$ be a random process. We define the temporal mean of $X_t$ as follows
\begin{align*}
\langle m_X\rangle_{_T}={1\over 2T}\, \int_{-T}^{T} X_t\, {\rm d} t, \quad T>0.
\end{align*}
\end{definition}

\begin{definition}
A process $X_t$  with mean $\mu$ independent of time is mean ergodic if
\begin{align*}
\lim_{T\to\infty} {\rm Var}(\langle m_X\rangle_{_T})
= 
\lim_{T\to\infty} 
\mathbb{E}(\langle m_X\rangle_{_T}-\mu)^2
=0.
\end{align*}
\end{definition}

\begin{proposition}
The BN random process  with mean $\mu$ independent of time is mean ergodic whenever 
\begin{align}\label{condition-ergodic}
\lim_{T\to\infty}   {1\over 2T}\, \int_{-T}^{T} \sigma_t\, {\rm d} t=0.
\end{align}

For example, we can take $\sigma_t=\exp({-t^2})$.
\end{proposition}
\begin{proof}
A simple calculus shows that
\begin{align*}
{\rm Var}(\langle m_X\rangle_{_T})
=
{1\over 4T^2}\,
\int_{-T}^{T} \int_{-T}^{T} C_X(t,t')\, {\rm d} t' {\rm d} t.
\end{align*}
Since $C_X(t,s)\stackrel{\eqref{cov}}{=}\sigma_t\sigma_{t'}[
\rho(1+\alpha^2)+(1-\rho^2)\alpha]$, it follows that
\begin{align*}
{\rm Var}(\langle m_X\rangle_{_T})
=
[
\rho(1+\alpha^2)+(1-\rho^2)\alpha]
\biggl(
{1\over 2T}\,
\int_{-T}^{T} \sigma_t\, {\rm d} t
\biggr)^2.
\end{align*}
Letting $T\to\infty$ in the above equality, from condition \eqref{condition-ergodic} the proof follows.
\end{proof}

\begin{definition}
	A WSS process $X_t$ is covariance-ergodic if
	\begin{align*}
	\lim_{T\to\infty} 
	{\rm Var}\biggl[
	{1\over 2T}\, \int_{-T}^{T} (X_t-\mu)(X_{t+s}-\mu)\, {\rm d} t
	\biggr]
	=0.
	\end{align*}
	
When $s=0$ the WSS process is called variance ergodic.
\end{definition}

In general, the BN random process $X_t$ is not a WSS process (see Proposition \ref{not-WSS}). Then it is clear that $X_t$ is not a covariance ergodic process. 

\begin{proposition}
When $\mu_t$ is independent of time and $\rho=\alpha=0$ the BN process is variance ergodic whenever
\begin{align}\label{cond-auto}
\lim_{T\to\infty} 
{1\over 4T^2}\,
\int_{-T}^{T} \int_{-T}^{T} {\rm Cov}(X_t^2,X_{t'}^2)\, {\rm d} t' {\rm d} t
=
\lim_{T\to\infty} 
{1\over 4T^2}\,
\int_{-T}^{T} \int_{-T}^{T} {\rm Cov}(X_t^2,X_{t'})\, {\rm d} t' {\rm d} t
=0.
\end{align}
\end{proposition}
\begin{proof}
When $\rho=\alpha=0$, $C_X(t,t')=0$. A simple calculus shows that
\begin{align*}
	{\rm Var}\biggl[
{1\over 2T}\, \int_{-T}^{T} (X_t-\mu)^2\, {\rm d} t
\biggr]
=
{1\over 4T^2}\,
\int_{-T}^{T} \int_{-T}^{T} {\rm Cov}[(X_t-\mu)^2,(X_{t'}-u)^2]\, {\rm d} t' {\rm d} t.
\end{align*}
Since $C_X(t,t')=0$ the above expression is
\begin{align*}
=
{1\over 4T^2}\,
\int_{-T}^{T} \int_{-T}^{T} 
\big[{\rm Cov}(X_t^2,X_{t'}^2)-2\mu{\rm Cov}(X_t^2,X_{t'})-2\mu{\rm Cov}(X_t,X_{t'}^2)\big] \, {\rm d} t' {\rm d} t.
\end{align*}
By using condition \eqref{cond-auto} the proof follows.
\end{proof}

\section{A triangular array central limit theorem}\label{sec:triangular}

\begin{definition}
Two random variables $X$ and $Y$ are said to be positively quadrant
dependent (PQD) if, for all $x,y\in\mathbb{R}$,
\begin{align*}
H(x,y)=\mathbb{P}(X>x,Y>y)-\mathbb{P}(X>x)\mathbb{P}(Y>y)\geq 0.
\end{align*}
\end{definition}
It is usual to rewrite $H(x,y)$ using distribution
functions as follows:
\begin{align}\label{ident-H}
H(x,y)=\mathbb{P}(X\leq x,Y\leq y)-\mathbb{P}(X\leq x)\mathbb{P}(Y\leq y).
\end{align}

\begin{remark}\label{rem-H}
If $F$ is a CDF, for all $x,y\in \mathbb{R}^2$ and $\alpha\in\mathbb{R}$, the following holds
\begin{align*}
F(\min\{x,y\}-\alpha)
+
F(\min\{x,y\}+\alpha)
\geq
{1\over 2}\, \big[
F(x-\alpha)
+
F(x+\alpha)
\big]
\big[
F(y-\alpha)
+
F(y+\alpha)
\big].
\end{align*}

Indeed, without loss of generality, assume that $x<y$. Then
\begin{align*}
F(\min\{x,y\}-\alpha)
+
F(\min\{x,y\}+\alpha)
&=
F(x-\alpha)
+
F(x+\alpha)
\\[0,2cm]
&\geq
{1\over 2}\, \big[
F(x-\alpha)
+
F(x+\alpha)
\big]
\big[
F(y-\alpha)
+
F(y+\alpha)
\big],
\end{align*}
because 
$0\leq F(y-\alpha)
+
F(y+\alpha)\leq 2$.
\end{remark}

By stochastic representation of Proposition \ref{Stochastic representation}, if $X_j\sim \text{BN}(\boldsymbol{\theta}_j)$, there are $Z_j\sim N(0,1)$ and $A_j\sim {\rm Bernoulli}(1/2)$, with $A_j\in\{\pm\alpha\}$, so that $X_j=\sigma_j(Z_j+A_j)+\mu_j$. From now on, in this section, we assume that variables $Z_j$ and $A_j$ are independent of $j$. I.e., 
\begin{align}\label{rep-stoch}
X_j=\sigma_j(Z+A)+\mu_j.
\end{align}

\begin{proposition}\label{Prop-pqd}
The random variables $X\sim \text{BN}(\boldsymbol{\theta}_X)$ and $Y\sim \text{BN}(\boldsymbol{\theta}_Y)$ are PQD, where
 $\boldsymbol{\theta}_X=(\mu_X,\sigma_X,\alpha)$, $\mu_X\in\mathbb{R}$, $\sigma_X>0$ and $\alpha\in \mathbb{R}$.
\end{proposition}
\begin{proof}
By \eqref{rep-stoch}, $X=\sigma_X(Z+A)+\mu_X$ and $Y=\sigma_Y(Z+A)+\mu_Y$.
Then
\begin{align*}
\mathbb{P}(X\leq x,Y\leq y)
&=
\mathbb{P}\biggl(Z\leq {x-\mu_X\over\sigma_X}-A ,Z\leq {y-\mu_Y\over\sigma_Y}-A\biggr)
\\[0,2cm]
&=
\mathbb{P}\biggl(Z\leq \min\Big\{{x-\mu_X\over\sigma_X},{y-\mu_Y\over\sigma_Y}\Big\}-A\biggr)
=
\mathbb{P}(Z\leq \varphi_{n;t_0}(A)).
\end{align*}
Let $\widehat{\mathbb{E}}$ be the expectation over $A$. By Fubini’s theorem we have
\begin{align*}
\mathbb{P}(Z\leq \varphi_{n;t_0}(A))
&=
\widehat{\mathbb{E}}\biggl[\Phi\biggl(\min\Big\{{x-\mu_X\over\sigma_X},{y-\mu_Y\over\sigma_Y}\Big\}-A\biggr)\biggr]
\\[0,2cm]
&=
{1\over 2}\, \biggl[
\Phi\biggl(\min\Big\{{x-\mu_X\over\sigma_X},{y-\mu_Y\over\sigma_Y}\Big\}-\alpha\biggr)
+
\Phi\biggl(\min\Big\{{x-\mu_X\over\sigma_X},{y-\mu_Y\over\sigma_Y}\Big\}+\alpha\biggr)
\biggr].
\end{align*}
Therefore,
\begin{align*}
\mathbb{P}(X\leq x,Y\leq y)
=
{1\over 2}\, \biggl[
\Phi\biggl(\min\Big\{{x-\mu_X\over\sigma_X},{y-\mu_Y\over\sigma_Y}\Big\}-\alpha\biggr)
+
\Phi\biggl(\min\Big\{{x-\mu_X\over\sigma_X},{y-\mu_Y\over\sigma_Y}\Big\}+\alpha\biggr)
\biggr].
\end{align*}

On the other hand, by using the identity 
${\rm erf}({x/\sqrt{2}})=2\Phi(x)-1$, the CDF \eqref{CDF-Bigaussian} of $X\sim \text{BN}(\boldsymbol{\theta}_X)$ is written as
\begin{align*}
\mathbb{P}(X\leq x)
=
{1\over 2}\, \biggl[
\Phi\biggl({x-\mu_X\over\sigma_X}-\alpha\biggr)
+
\Phi\biggl({x-\mu_X\over\sigma_X}+\alpha\biggr)
\biggr].
\end{align*}
Hence, by Remark \ref{rem-H} we get
\begin{align*}
H(x,y)
&\stackrel{\eqref{ident-H}}{=}
\mathbb{P}(X\leq x,Y\leq y)-\mathbb{P}(X\leq x)\mathbb{P}(Y\leq y)
\\[0,3cm]
&=
{1\over 2}\, \biggl[
\Phi\biggl(\min\Big\{{x-\mu_X\over\sigma_X},{y-\mu_Y\over\sigma_Y}\Big\}-\alpha\biggr)
+
\Phi\biggl(\min\Big\{{x-\mu_X\over\sigma_X},{y-\mu_Y\over\sigma_Y}\Big\}+\alpha\biggr)
\biggr]
\\[0,3cm]
&\quad -
{1\over 4}\, \biggl[
\Phi\biggl({x-\mu_X\over\sigma_X}-\alpha\biggr)
+
\Phi\biggl({x-\mu_X\over\sigma_X}+\alpha\biggr)
\biggr]
\biggl[
\Phi\biggl({y-\mu_Y\over\sigma_Y}-\alpha\biggr)
+
\Phi\biggl({y-\mu_Y\over\sigma_Y}+\alpha\biggr)
\biggr]
\\[0,3cm]
&
\geq 0.
\end{align*}
This completes the proof.
\end{proof}

\begin{definition}
We define a sequence of random variables $\{X_j\}$ to be
linearly positive quadrant dependent (LPQD) if for any disjoint $A,B$ and positive $\{\lambda_j\}$, $\sum_{k\in A} \lambda_k X_k$
and $\sum_{l\in B} \lambda_l X_l$ are PQD.
\end{definition}

A reasoning similar to the proof of 
Proposition \ref{Prop-pqd} gives the following result.
\begin{proposition}\label{Prop-lpqd}
	The sequence of random variables $\{X_j\}$, with $X_j\sim \text{BN}(\boldsymbol{\theta}_j)$, is LPQD, where
	$\boldsymbol{\theta}_j=(\mu_j,\sigma_j,\alpha)$, $\mu_j\in\mathbb{R}$, $\sigma_j>0$ and $\alpha\in \mathbb{R}$.
\end{proposition}

\begin{theorem}[Triangular array CLT]
Let $S_n=\sum_{j=1}^{M_n}[X_{n,j}-\mathbb{E}(X_{n,j})]$ 
where for each $n$, $X_{n,j}\sim\text{BN}(\boldsymbol{\theta}_{n,j})$, with
$\boldsymbol{\theta}_{n,j}=(\mu_{n,j},\sigma_{n,j},\alpha)$, $\mu_{n,j}\in\mathbb{R}$, $\sigma_{n,j}>0$ and $\alpha\in \mathbb{R}$.
Suppose there exist $c_1,c_2,c_3\in(0,\infty)$ and a sequence $u_l\to 0$ so that for all $n,j,l$, the following hold:
\begin{align}
& \sigma_{n,j}^2\geq c_1, \ \sigma_{n,j}^3\leq c_2; \label{Assumptions}
\\[0,1cm]
& \sum_{k=1}^{M_n} \sigma_{n,j}\sigma_{n,k}\leq c_3; \label{Assumptions-1}
\\[0,1cm]
&
\sum_{\substack{k=1 \\ \vert k-j\vert\geq l}}^{M_n} \sigma_{n,j}\sigma_{n,k} \leq u_l; \label{Assumptions-2}
\end{align}	
	then
	\begin{align*}
\lim_{n\to\infty}
\mathbb{P}\big([{\rm Var}(S_n)]^{-1/2} S_n \leq x\big)
=
{1\over \sqrt{2\pi}}\, \int_{-\infty}^x \exp(-x^2/2) \, {\rm d} x, 
\quad \forall x\in\mathbb{R}.
	\end{align*}
\end{theorem}
\begin{proof}
	Since for each $n$, $\{X_{n,j}\}$ is LPQD (see Proposition \ref{Prop-lpqd}) but not SSS (see Proposition \ref{not-SSS}),
by \cite{cox:02} it is enough to verify that:
\begin{align}
& {\rm Var}(X_{n,j})\geq \widetilde{c}_1, \ \mathbb{E}\vert X_{n,j}-\mathbb{E}(X_{n,j})\vert^3\leq \widetilde{c}_2; \label{Assumptions-to-verify}
\\[0,1cm]
& \sum_{k=1}^{M_n} {\rm Cov}(X_{n,j},X_{n,k})\leq \widetilde{c}_3; \label{Assumptions-to-verify-1}
\\[0,1cm]
&
\sum_{\substack{k=1 \\ \vert k-j\vert\geq l}}^{M_n} {\rm Cov}(X_{n,j},X_{n,k})\leq \widetilde{u}_l; \label{Assumptions-to-verify-2}
\end{align}
where  $\widetilde{u}_l \to 0$.

Indeed, since, by \eqref{Assumptions},  $\sigma^2_{n,j}\geq c_1$ and ${\rm Var}(X_{n,j})=\sigma^2_{n,j}(1+\alpha^2)$ (see property P.8 in Section \ref{sec:pre:prop}) we have ${\rm Var}(X_{n,j}) \geq\sigma^2_{n,j}\geq c_1=\widetilde{c}_1$. Moreover, using the representation in \eqref{rep-stoch} and the condition \eqref{Assumptions} we obtain 
\begin{align*}
\mathbb{E}\vert X_{n,j}-\mathbb{E}(X_{n,j})\vert^3
=
\sigma^3_{n,j}
  \mathbb{E}\vert Z+A\vert^3
  \leq 
  6\sqrt{2/\pi}\, \sigma^3_{n,j}
    \leq 
 5c_2=\widetilde{c}_2.
\end{align*}
That is, \eqref{Assumptions-to-verify} is satisfied.

On the other hand, since ${\rm Cov}(X_{n,j},X_{n,k})
\stackrel{\eqref{cov}}{=}
\sigma_{n,j}\sigma_{n,k}[
\rho(1+\alpha^2)+(1-\rho^2)\alpha]$, by conditions  \eqref{Assumptions-1} and \eqref{Assumptions-2}, the statements in 
\eqref{Assumptions-to-verify-1} and \eqref{Assumptions-to-verify-2} follow by taking $\widetilde{c}_3=c_3 [
\rho(1+\alpha^2)+(1-\rho^2)\alpha]$ and $\widetilde{u}_l= [
\rho(1+\alpha^2)+(1-\rho^2)\alpha] u_l$, respectively. 
\end{proof}

\begin{remark}
The set of $\sigma_{n,k}$'s satisfying conditions \eqref{Assumptions}, \eqref{Assumptions-1} and \eqref{Assumptions-2} is non-empty. 

Indeed, let us take $M_n=n$  and  $\sigma_{n,k}=r^{-k}$,
$k\geq 1$, $r>1$, for all $n$. 
Immediately, we have $\sigma_{n,k}>0$ and $\sigma_{n,k}\leq 1$.
That is, \eqref{Assumptions} is valid.
Moreover,
\begin{align*}
\sum_{k=1}^{n} \sigma_{n,j}\sigma_{n,k}
\leq 
\sum_{k=1}^{n} \sigma_{n,k}
\leq 
\sum_{k=1}^{\infty} {1\over r^k}
=
{1\over r-1}, 
\quad \text{for} \ r>1.
\end{align*}
%
Then \eqref{Assumptions-1} is satisfied.

Finally, since  $r^{\vert k-j\vert}\leq r^{j+k}$ for $r>1$, we have
 %
$
\sigma_{n,j}\sigma_{n,k}
=
r^{-(j+k)}
\leq 
r^{-\vert k-j\vert}.
$
Then,
\begin{align*}
\sum_{\substack{k=1 \\ \vert k-j\vert\geq l}}^{n} 
\sigma_{n,j} \sigma_{n,k}
\leq 
\sum_{\substack{k=1 \\ \vert k-j\vert\geq l}}^{n} 
{1\over r^{\vert k-j\vert}}
\leq
\sum_{\substack{k=1 \\ \vert k-j\vert\geq l}}^{\infty} 
{1\over r^{\vert k-j\vert}}
=
\sum_{m=l}^{\infty}
{1\over r^{m}}
\left[ \,
\sum_{\substack{k=1 \\ \vert k-j\vert= m}}^{\infty} 
1
\right],
\end{align*}
where in the last equality we rearrange the summation terms. Since
$\big[ \,
\sum_{\substack{k: \vert k-j\vert= m}}1 \big]$
is the number of vertices at the boundary of the one-dimensional
 ball of
radius $m$ centered at $j$, there is $C> 0$ independent of $j$ such that $\big[ \,
\sum_{\substack{k: \vert k-j\vert= m}}1 \big]=C$. Hence
\begin{align*}
\sum_{\substack{k=1 \\ \vert k-j\vert\geq l}}^{n} 
\sigma_{n,j} \sigma_{n,k}
\leq 
C \sum_{m=l}^{\infty}
{1\over r^{m}} = u_l.
\end{align*}
Since $\sum_{m=0}^{\infty}{r^{-m}}=r(r-1)^{-1}<\infty$ for $r>1$, it follows that $u_l\to 0$ when $l\to \infty$.
Therefore, \eqref{Assumptions-2} follows.
\end{remark}

\section{Numerical evaluation}\label{sec:simulation}
In this section, a Monte Carlo simulation study was carried out to evaluate the performance of the 
maximum likelihood estimators of the BN model; see Section \ref{sec:asymptot}. All numerical evaluations were done in
the \texttt{R} software; see \cite{r:18}. 

The simulation scenario considers sample size $n \in \{10,75,250,600\}$, location parameter $\mu=\{0.5\}$, scale parameter $\sigma=\{1.0\}$, location parameter $\alpha \in \{-2.0,-0.5,0.8,3.0\}$, with 1,000 Monte Carlo replications for each combination of above given parameters and sample size. The values of the location parameter $\alpha$ have been chosen in order to study the performance under uni- and bimodality.

The maximum likelihood estimation results for the considered BN model are presented in Figure \ref{fig:estimates}. The empirical bias and root mean squared error (MSE) are reported. A look at the results in in Figure \ref{fig:estimates} allows us to conclude that, as the sample size increases, the empirical bias and RMSE both decrease, as expected. Moreover, we note that the performance of the estimate of $\mu$ is better when $|\alpha|>1$, namely, under bimodality.

\begin{figure}[!ht]
\centering
\subfigure[$\alpha=-2.0$]{\includegraphics[height=5.5cm,width=5.5cm]{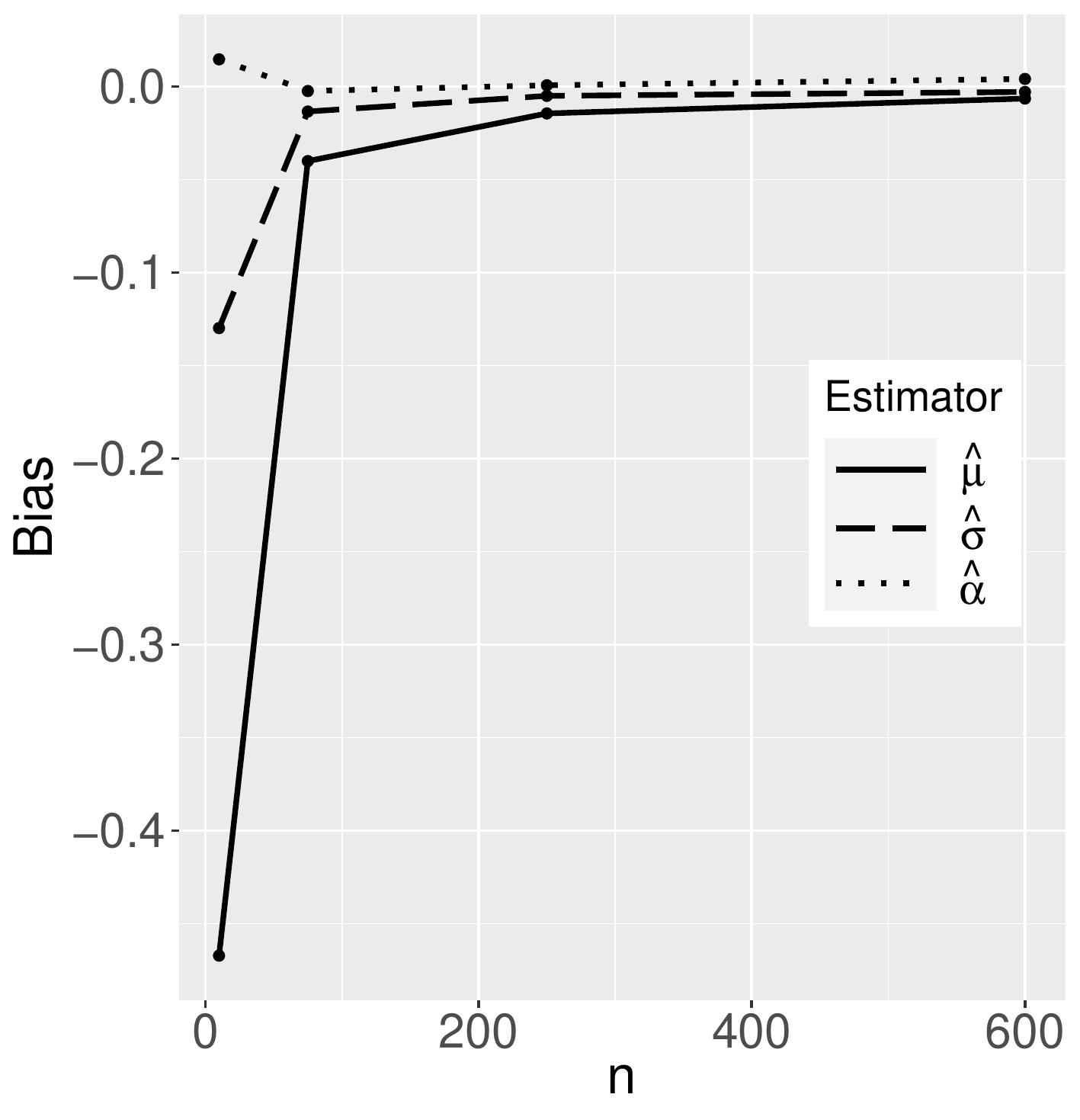}}
\subfigure[$\alpha=-2.0$]{\includegraphics[height=5.5cm,width=5.5cm]{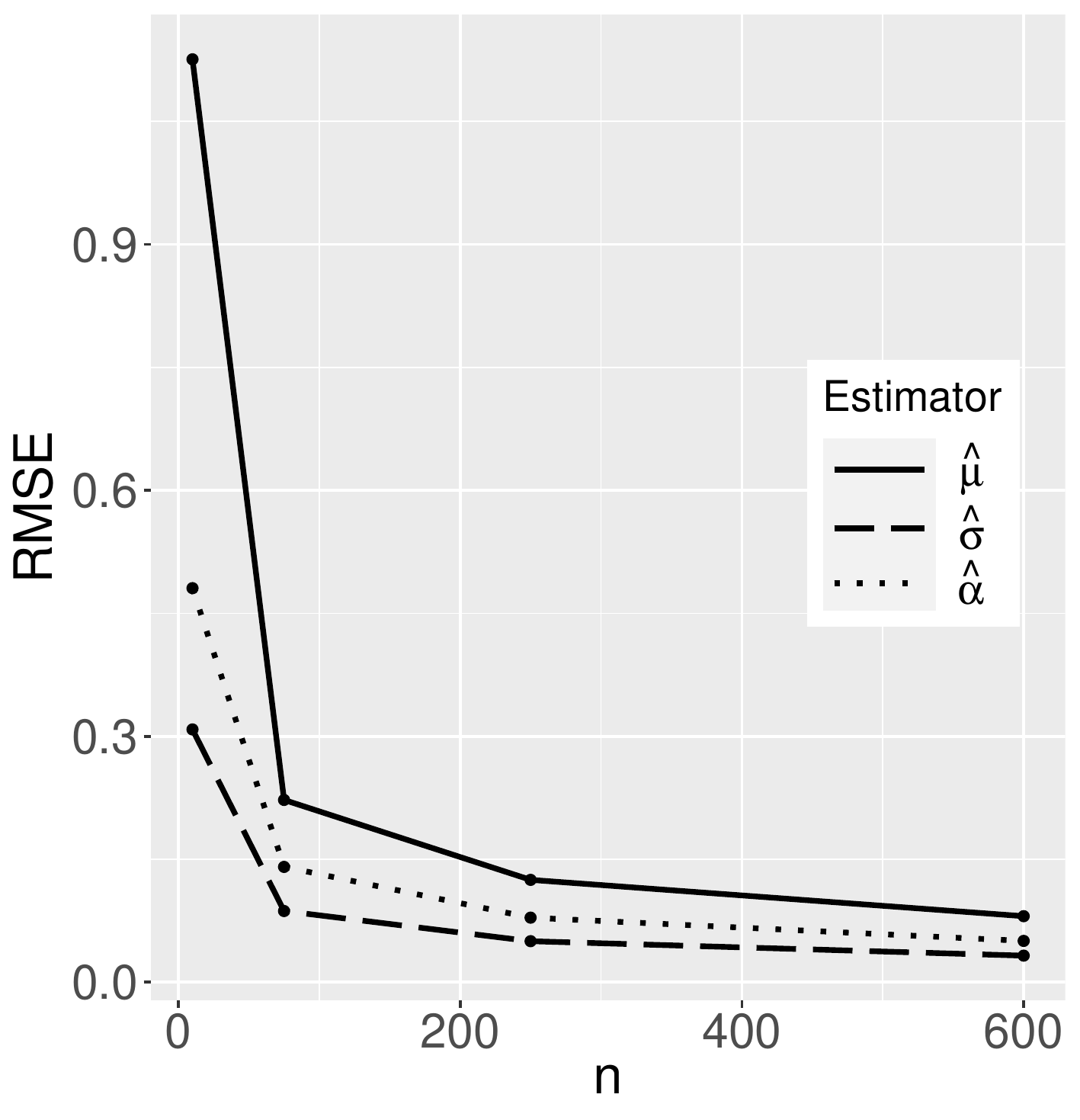}}
\subfigure[$\alpha=-0.5$]{\includegraphics[height=5.5cm,width=5.5cm]{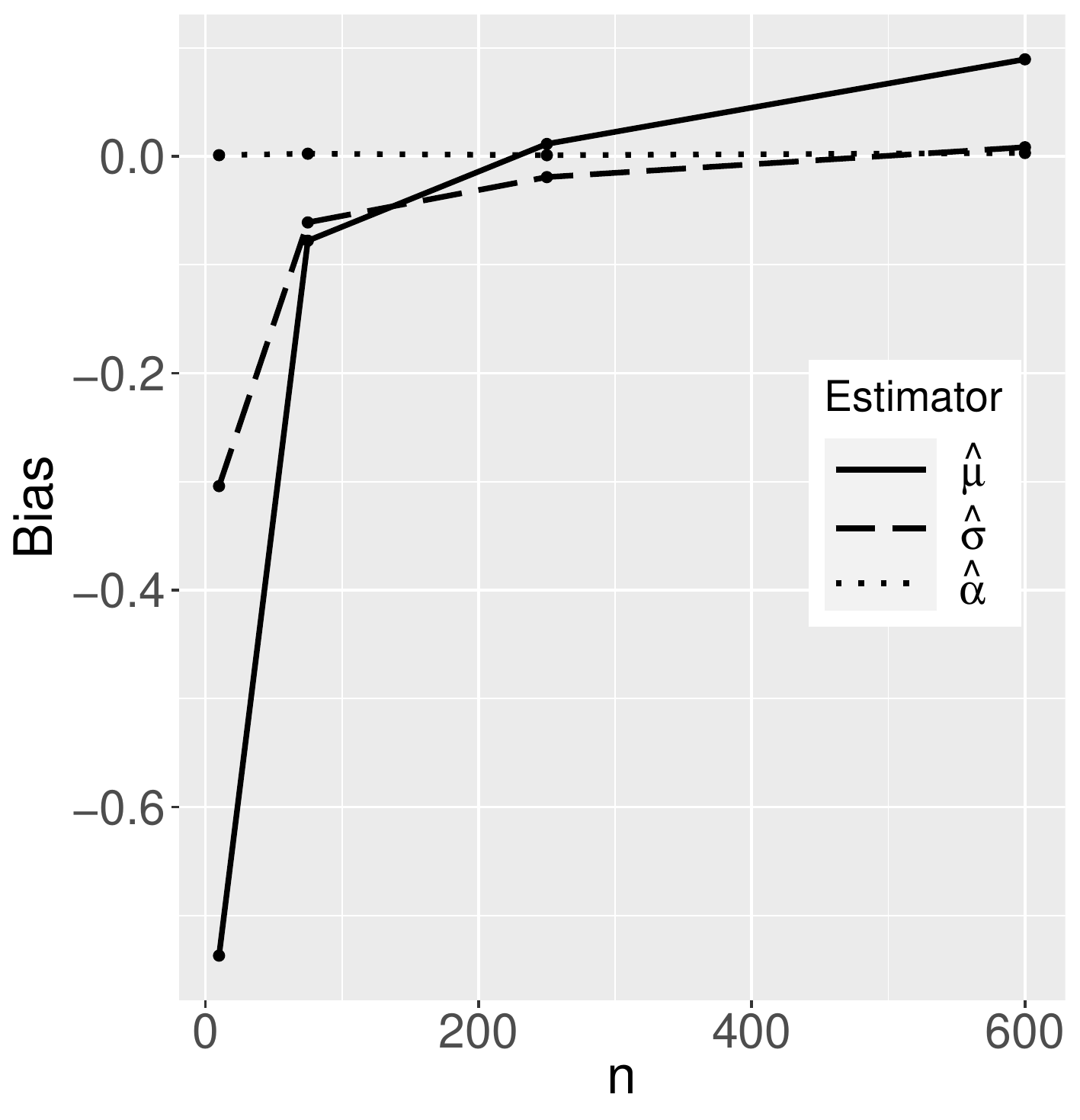}}\\
\subfigure[$\alpha=-0.5$]{\includegraphics[height=5.5cm,width=5.5cm]{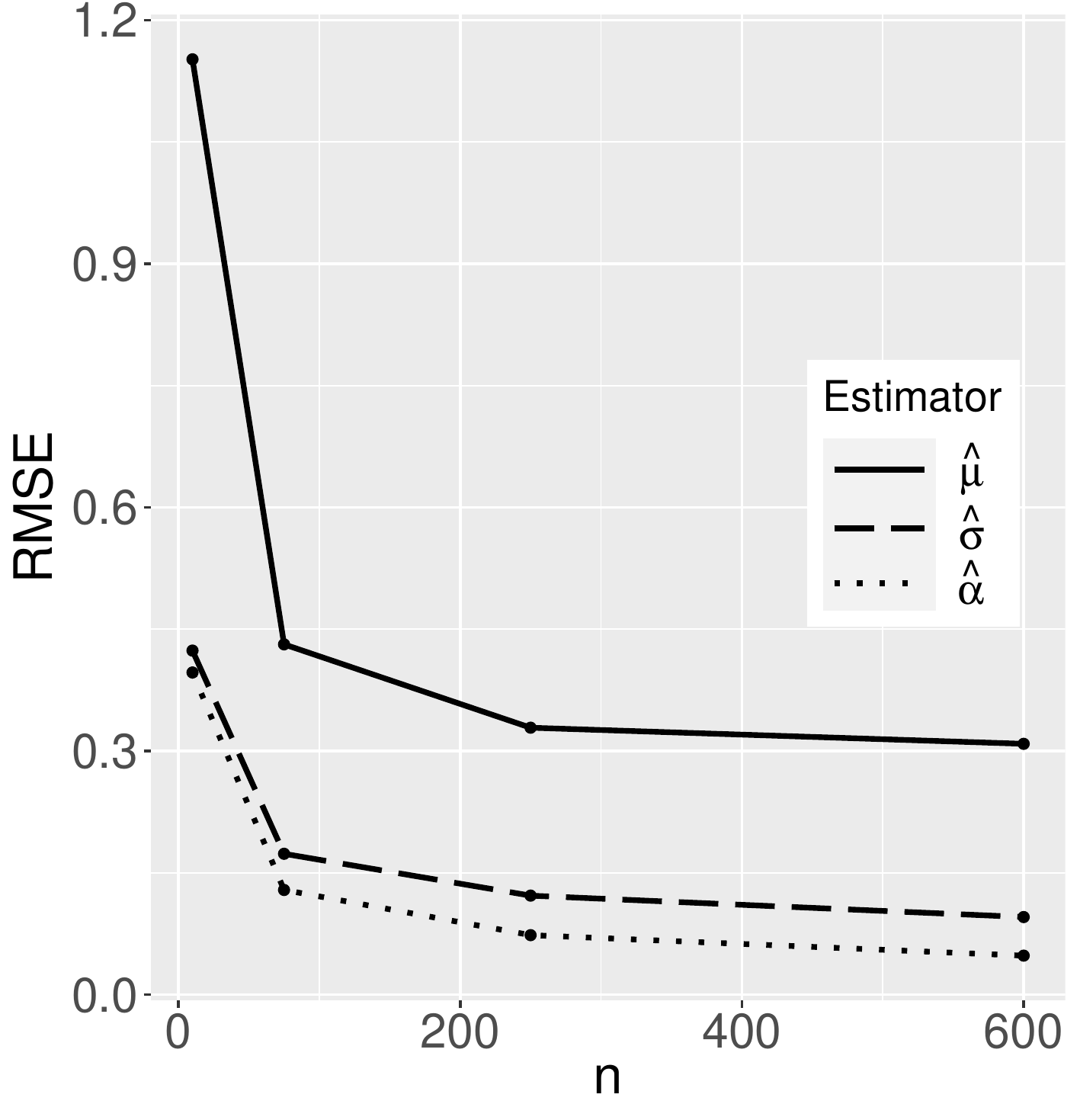}}
\subfigure[$\alpha=0.8$]{\includegraphics[height=5.5cm,width=5.5cm]{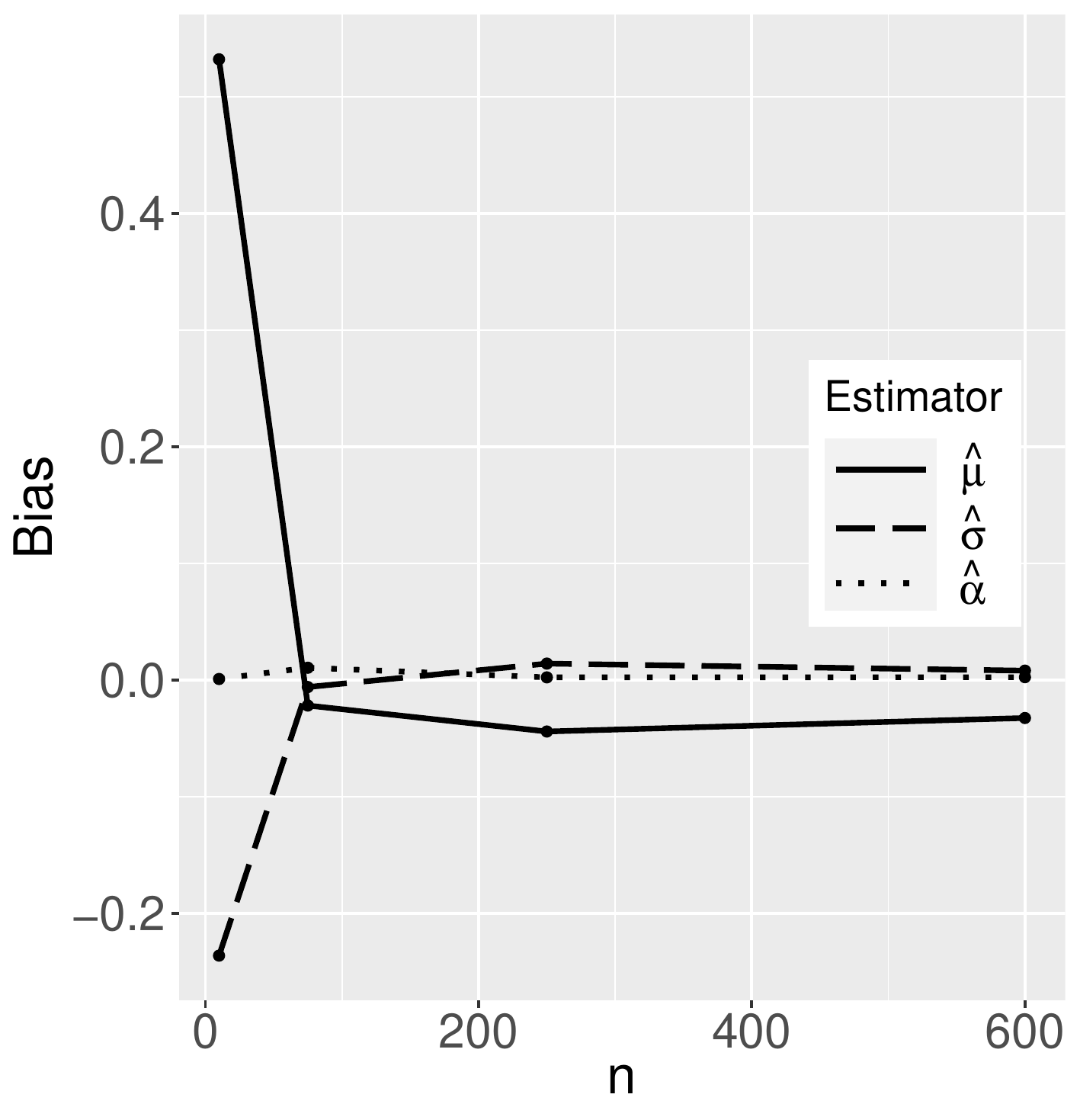}}\\
\subfigure[$\alpha=0.8$]{\includegraphics[height=5.5cm,width=5.5cm]{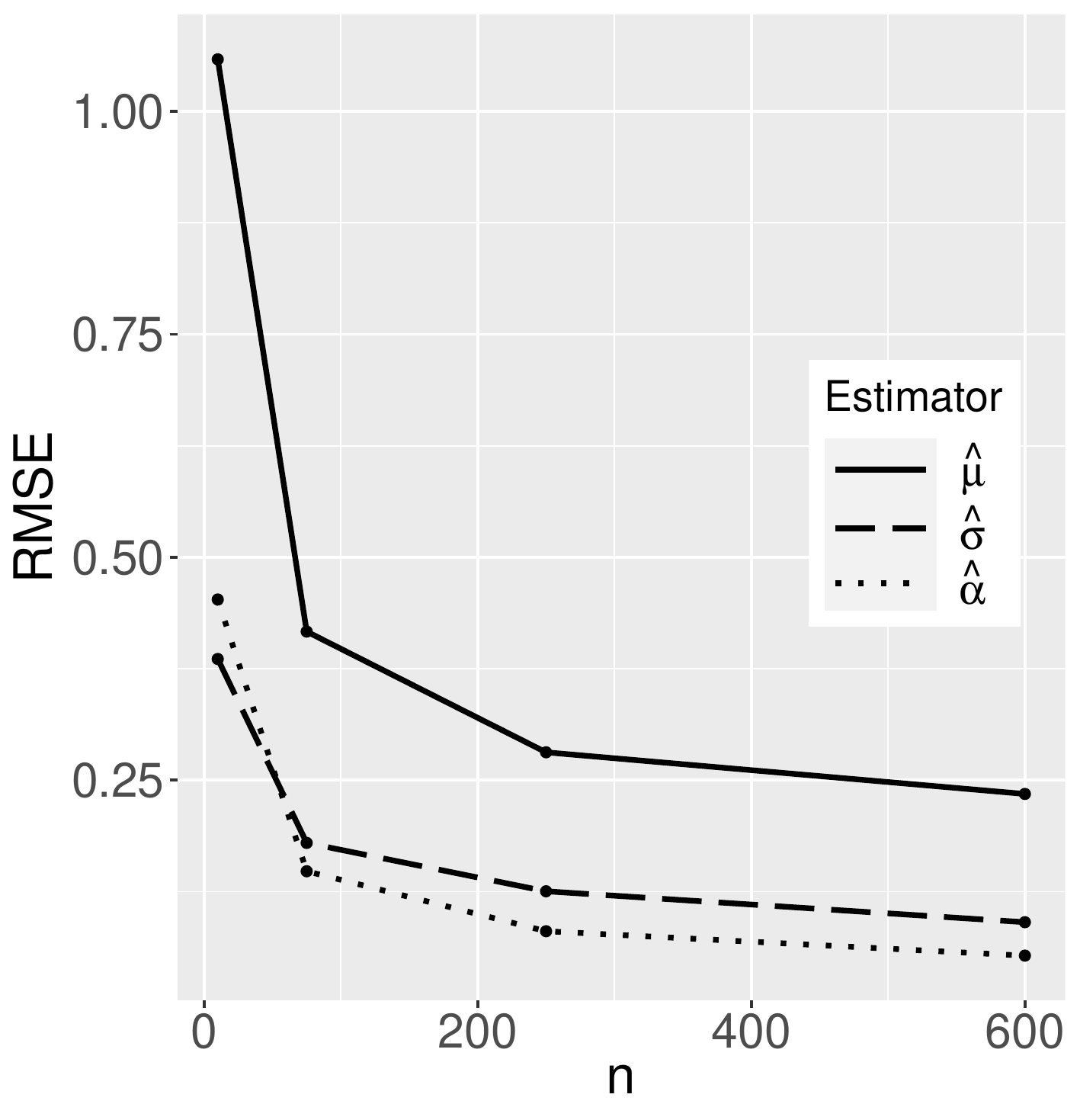}}
\subfigure[$\alpha=3.0$]{\includegraphics[height=5.5cm,width=5.5cm]{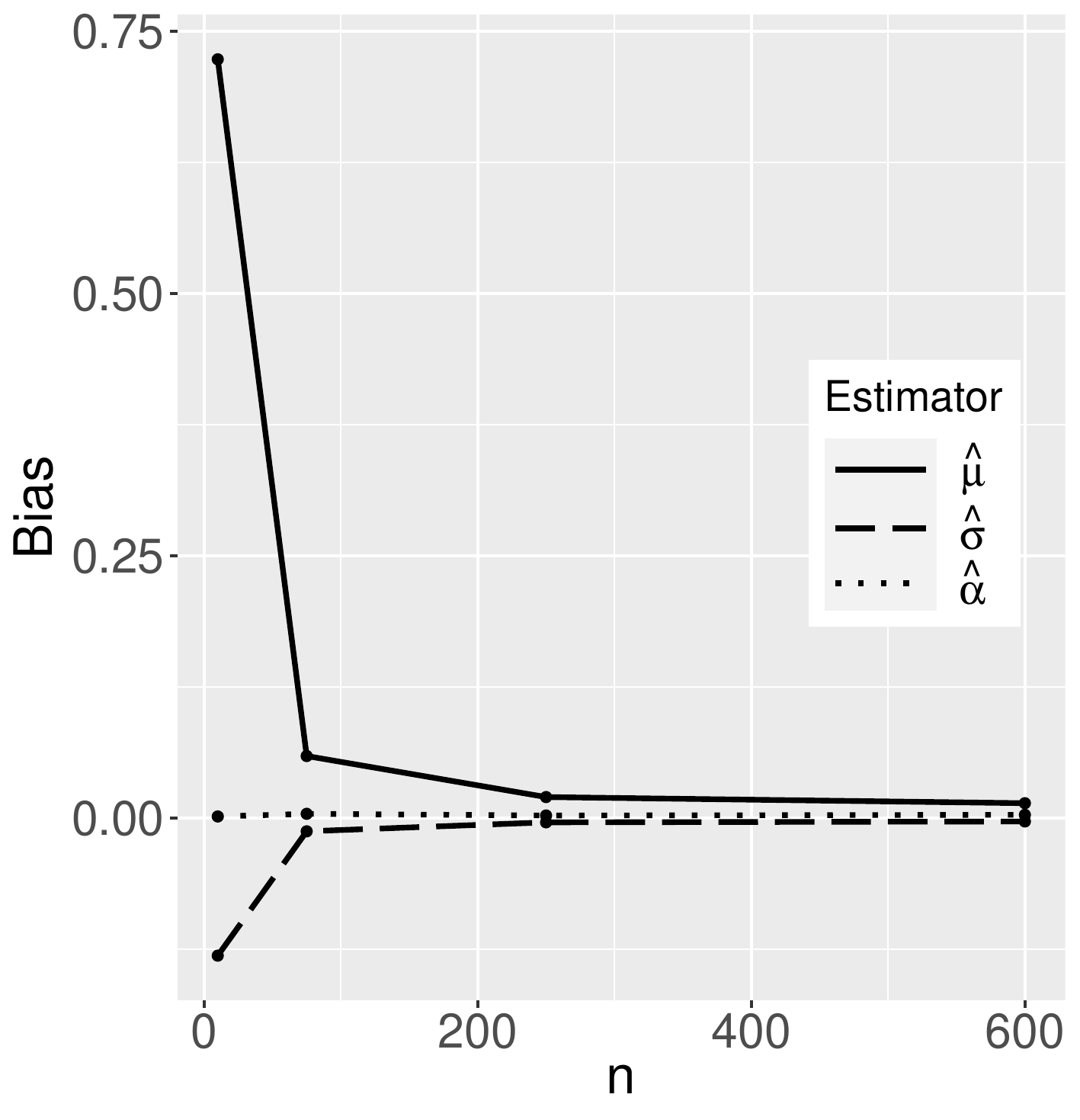}}
\subfigure[$\alpha=3.0$]{\includegraphics[height=5.5cm,width=5.5cm]{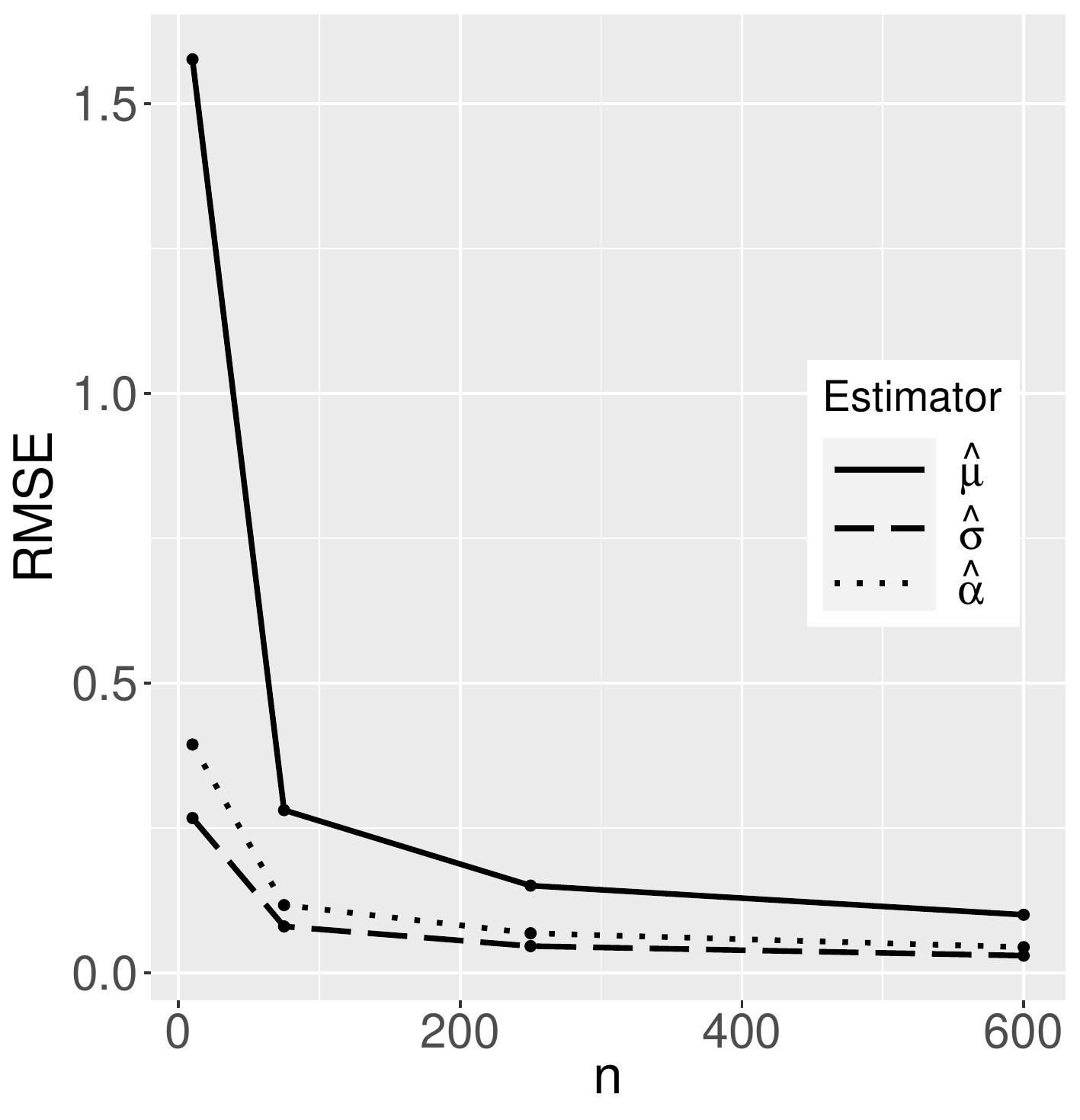}}\\
 \caption{\small {Empirical bias and RMSE from simulated data for the indicated maximum likelihood estimates of the BN 
model parameters, $n$ and $\alpha$.}}
\label{fig:estimates}
\end{figure}

\section{Concluding remarks}\label{sec:concluremarks}
We have derived novel properties of the bimodal normal distribution. We have discussed some mathematical properties, proof for the bimodality and identifiability. We have also discussed some aspects related to the maximum likelihood estimation as well as associated asymptotic properties. We have derived a bivariate version of the bimodal normal distribution and analyzed some characteristics such as covariance and correlation. We have studied stationarity and ergodicity and a triangular array central limit theorem. Finally, we have carried out Monte Carlo simulations to evaluate the behaviour of the maximum likelihood estimates.

\clearpage

\paragraph{Acknowledgements}
The authors thank CNPq for the financial support. This study was financed in part by the Coordenação de Aperfeiçoamento de Pessoal de Nível Superior - Brasil (CAPES) (Finance Code 001).

\paragraph{Disclosure statement}
There are no conflicts of interest to disclose.


\bibliographystyle{apalike}

\small
\begin{thebibliography}{}
%
\bibitem[Azzalini and Bowman(1990)]{ab:90}
{Azzalini, A. and Bowman, A. W.} (1990) 
A look at some data on the Old Faithful geyser. 
\emph{Applied Statistics}, {39}, 357--365. 
%
\bibitem[Bahadur(1971)]{baj:71}
Bahadur, R. R. (1971). 
\emph{Some Limit Theorems in Statistics}. 
SIAM, Philadelphia.

\bibitem[Cox and Grimmett (1984)]{cox:02}
Cox, J. T., and  Grimmett, G. (1984).
\emph{Central Limit Theorems for Associated Random Variables and the Percolation Model}.
Ann. Probab. 12(2): 514-528. 

\bibitem[Cramér(1946)]{cra:46}
Cramér, H. (1946). 
\emph{Mathematical methods of statistics}. 
NJ, US: Princeton University Press.

\bibitem[Eugene et al.(2002)]{eugeneetal:02}
Eugene, N., Lee, C., and Famoye, F. (2002). 
\emph{Beta-normal distribution and its applications}. 
Communications in Statistics - Theory and Methods, 31, 497-512.

\bibitem[G\'omez-D\'eniz et al.(2021)]{gomezdenizetal:21}
G\'omez-D\'eniz, E., Sarabia, J. M., and Calderín-Ojeda, E. (2021). 
\emph{Bimodal normal distribution: Extensions and applications}. 
Journal of Computational and Applied Mathematics, 388, 113292.

\bibitem[Hassan and El-Bassiouni(2016)]{hassanelbassiouni:16}
Hassan, M.Y. and El-Bassiouni, M.Y. (2016). 
\emph{Bimodal skew-symmetric normal distribution}. 
Communications in Statistics - Theory and Methods, 45, 1527-1541.


\bibitem[R-Team, 2020]{r:18}
R-Team (2020).
\newblock {\em {R: A Language and Environment for Statistical Computing}}.
\newblock R Foundation for Statistical Computing, Vienna, Austria.

\bibitem[Teicher(1963)]{Teicher:63}
Teicher, H. (1963). 
\emph{Identifiability of Finite Mixtures}. 
The Annals of Mathematical Statistics, 34(4), 1265-1269.

\bibitem[Vila et al.(2020)]{vilaetal:20}
Vila, R., Le\~ao, J., Saulo, H., Shahzad, M. N., and Santos-Neto, M. (2020). 
\emph{On a bimodal Birnbaum-Saunders distribution with applications to lifetime data}. 
Brazilian Journal of Probability and Statistics, 34, 495-518.

\bibitem[Winkelbauer(2014)]{Win:14}
Winkelbauer, A. (2014). 
\emph{Moments and absolute moments of the normal distribution}.
Preprint. ArXiv:1209.4340.

 
\end{thebibliography}

\end{document}